\documentclass[12pt,reqno]{amsart}

\usepackage[svgnames]{xcolor}
\usepackage[centering, margin = 50pt]{geometry}
\usepackage[bbgreekl]{mathbbol}
\usepackage[new]{old-arrows}

\usepackage{amsmath, amssymb, amsthm}
\usepackage{amsfonts}
\usepackage{mathtools}

\usepackage{euscript}
\usepackage{enumerate, url}
\usepackage{enumitem}\setlist[enumerate,1]{ref=(\roman*)}
\usepackage{comment}
\usepackage{tikz-cd}
\usepackage{quiver}
\usepackage{mathrsfs}
\usepackage{graphicx}
\usepackage[colorlinks,citecolor=DarkRed,linkcolor=blue,urlcolor=blue]{hyperref}
\usepackage{stmaryrd}
\SetSymbolFont{stmry}{bold}{U}{stmry}{m}{n}

\usepackage[style=alphabetic]{biblatex}
\usepackage{soul}
\setul{3pt}{.4pt}

\usetikzlibrary{arrows.meta}
\newcommand{\ultrato}{\mathrel{\tikz[>={To}, line cap=round, anchor=base] \draw[-<] (0,0) -- (3.375mm,0);}}
\newcommand{\longultrato}{\mathrel{\tikz[>={To}, line cap=round, anchor=base] \draw[-<] (-13mm,0) -- (5mm,0);}}

\swapnumbers
\theoremstyle{plain}
\newtheorem{Thm}{Theorem}[section]
\newtheorem*{nThm}{Theorem}
\newtheorem{prop}[Thm]{Proposition}
\newtheorem{cor}[Thm]{Corollary}
\newtheorem{lem}[Thm]{Lemma}
\theoremstyle{definition}
\newtheorem{defn}[Thm]{Definition}
\newtheorem{cons}[Thm]{Construction}
\newtheorem{nota}[Thm]{Notation}
\newtheorem*{nnota}{Notation}
\newtheorem{ex}[Thm]{Example}
\newtheorem{rmk}[Thm]{Remark}
\newtheorem*{nrmk}{Remark}

\newtheorem*{nclaim}{Claim}
\newtheorem{question}[Thm]{Question}

\DeclareMathOperator{\dom}{dom}

\DeclareMathOperator{\coeq}{coeq}
\DeclareMathOperator*{\Nat}{Nat}
\DeclareMathOperator*{\colim}{colim}

\newcommand{\bbT}{\mathbb{T}}
\newcommand{\bbN}{\mathbb{N}}

\newcommand{\C}{\mathscr{C}}
\newcommand{\E}{\mathscr{E}}
\newcommand{\F}{\mathscr{F}}

\newcommand{\M}{\mathscr{M}}
\newcommand{\N}{\mathscr{N}}

\newcommand{\K}{\mathscr{K}}

\newcommand{\eF}{\EuScript{F}}

\newcommand{\eI}{\EuScript{I}}

\newcommand{\eZ}{\EuScript{Z}}
\newcommand{\eX}{\EuScript{X}}
\newcommand{\eY}{\EuScript{Y}}

\newcommand{\ptuf}{\star}
\newcommand{\bb}{\bbbeta}
\newcommand{\cleq}{\preccurlyeq}
\newcommand{\cgeq}{\succcurlyeq}

\newcommand{\Sh}{\mathsf{sh}}
\newcommand{\vSh}{\mathfrak{sh}}
\newcommand{\Desc}{\mathsf{Desc}}
\newcommand{\GTop}{\mathsf{GTop}}
\newcommand{\CohTop}{\mathsf{CohTop}}

\newcommand{\vUlt}{\mathsf{vUlt}}
\newcommand{\vUltb}{\mathsf{vUlt}_{\mathrm{bounded}}}
\newcommand{\UMat}{\mathsf{UMat}}
\newcommand{\op}{\mathrm{op}}
\newcommand{\amp}{\mathsf{amp}}
\newcommand{\In}{\mathsf{in}}
\newcommand{\Cat}{\mathsf{Cat}}
\newcommand{\CAT}{\mathsf{CAT}}
\newcommand{\TopSp}{\mathsf{TopSp}}
\newcommand{\Ult}{\mathsf{Ult}}
\newcommand{\UltL}{\Ult^{\mathsf{L}}}
\newcommand{\pt}{\mathsf{pt}}
\newcommand{\vpt}{\mathfrak{pt}}
\newcommand{\id}{\mathsf{id}}
\newcommand{\ev}{\mathsf{ev}}
\newcommand{\Set}{\mathsf{Set}}

\newcommand{\SetO}{\mathsf{Set}[\mathbb{O}]}
\newcommand{\UF}{\EuScript{U}\EuScript{F}}
\newcommand{\Pretop}{\mathsf{Pretop}}
\newcommand{\Latt}{\mathsf{Latt}}
\newcommand{\Ind}{\mathsf{Ind}}

\newcommand{\Ob}{\mathsf{Ob}}
\newcommand{\Hom}{\mathsf{Hom}}
\newcommand{\Mod}{\mathsf{Mod}}
\newcommand{\cl}{\mathsf{cl}}
\newcommand{\Alex}{\mathsf{Alex}}
\newcommand{\res}[1]{{\upharpoonright{#1}}}
\newcommand{\ie}{\emph{i.e. }}
\newcommand{\paronto}{\rightharpoonup\mathrel{\mspace{-15mu}}\rightharpoonup}
\newcommand{\li}{\mathsf{h}}
\newcommand{\lii}{\mathsf{k}}

\DeclareRobustCommand{\et}{\text{ét}}

\usepackage{orcidlink}

\title{Extending conceptual completeness via virtual ultracategories}
\author{Gabriel Saadia\orcidlink{0009-0009-2336-6252}}
\address{Department of Mathematics, Stockholm University, Sweden}
\email{gabrielsaadia@gmail.com \\ gabriel.saadia@math.su.se}

\begin{document}

\begin{abstract}
    We introduce the notion of virtual ultracategory. From a topological point of view, this notion can be seen as a categorification of relational $\beta$-modules. From a categorical point of view, virtual ultracategories generalize ultracategories in the same way that multicategories generalize monoidal categories. From a logical point of view, whereas the points of a coherent topos form an ultracategory, the points of an arbitrary topos form a virtual ultracategory. We then extend Makkai--Lurie's conceptual completeness: a topos with enough points can be reconstructed from its virtual ultracategory of points.
\end{abstract}
\maketitle


We introduce \emph{virtual ultracategories}, a notion categorifying the already-known notion of \emph{relational $\beta$-module}: one can axiomatize the notion of a topology on a set by prescribing the limit points of each ultrafilter instead of the open subsets, this gives the notion of relational $\beta$-modules described in \cite{Barr}, and this notion is equivalent to topological spaces. As noticed in \cite{GenMultiCat}, relational $\beta$-modules are generalized multicategories, and in the same way, virtual ultracategories can be seen as a generalized multicategorical extension of usual ultracategories.

Categorifying topological spaces is an old story, and (Grothendieck) topoi constitute the main objects playing this role; they categorify the properties satisfied by the poset of opens of a space: topoi are categories whose objects are thought of as the ‘‘categorified opens’’ of a ‘‘categorified space’’. It is well known that topoi also have a logical meaning: they classify geometric theories, and one can think of a topos as a categorified space whose points are the models of the classified geometric theory.

In the same way, we can see virtual ultracategories from a logical point of view: they generalize ultracategories. Ultracategories are categories with abstract ultraproduct operations mimicking the usual ultraproduct of models of a first-order theory. This notion was introduced by Makkai in \cite{MakkaiStone} to prove a \emph{conceptual completeness} result: the category of models of a coherent theory has an ultracategory structure that is enough to recover the theory, or equivalently, from a toposic perspective, the points of a coherent topos form an ultracategory with enough structure to recover the topos. In this work, we extend this result: any topos induces a virtual ultracategory structure on its points, and, assuming that the topos has enough points, the virtual ultracategory structure on its points is enough to recover the topos we started from.
\begin{nThm}
    For $\E$ a topos with enough points, there is an equivalence
    \[\E\stackrel{\sim}\longrightarrow \vUlt(\vpt(\E),\Set)\]
    providing a way to reconstruct the topos $\E$ from its virtual ultracategory of points $\vpt(\E)$.
\end{nThm}

We would also like to emphasize the 0-dimensional case of our result, given by Theorem \ref{Thm:étale-ucvg} characterizing étale maps between topological spaces through ultraconvergence. 
We state this characterization and prove it using only elementary concepts on topological spaces, as we believe it has its own interest outside the theory of virtual ultracategories and topoi. 
\begin{nThm}
    A continuous map of topological spaces $p : E \to T$ is étale if and only if
    \begin{enumerate}
        \item each ultrafilter $\mu\cgeq p(e)$ has a unique lift $\nu\cgeq e$, and
        \item each ultrafilter $\mu\cgeq p(e)$ has a lift $\nu\cgeq e$ that is principal over $T$ (\ie $p$ is injective on a $\nu$-large set).
    \end{enumerate}
\end{nThm}

\subsection*{Similar works}
After the writing of the current paper the author got aware that the same result has been obtained at the same time and independently by van Gool, Marquès, and Tarantino in \cite{Sam} and by Hamad in \cite{Hamad2}, with, respectively, ``ultraconvergence spaces'' and ``generalized ultracategories'' playing the role of virtual ultracategories. This paper does not establish any connections with either of these works.

\tableofcontents

\section*{Introduction}

\subsection{Makkai's conceptual completeness}
In \cite{MakkaiStone}, Makkai proved a \emph{conceptual completeness} result for coherent logic, that is, a result providing a method to recover a coherent theory from a suitable structure on the collection of its models.
The categorical notion of theory used by Makkai is that of a \emph{small pretopos}, seen as classifying categories for coherent theories. We refer to \cite[Chapter 7]{MakkaiReyes} for an overview of the relation between small pretopoi and coherent theories; we will not distinguish between the two, and identify a coherent theory with its classifying pretopos.

The models of a small pretopos $\bbT$ are given by pretopos functors from $\bbT$ into the large pretopos of small sets
\[\Mod(\bbT) \coloneqq \Pretop(\bbT,\Set).\]
An object of $\bbT$ is thought of as a sort or a formula of the theory; and a model, seen as a pretopos functor into $\Set$, sends a formula to its \emph{interpretation} in the model.
As taking models corresponds to mapping into $\Set$, we have a natural map from $\bbT$ into its ‘‘bidual’’ given by mapping the category of models again into $\Set$
\[\ev : \bbT\longrightarrow\CAT(\Mod(\bbT),\Set),\]
for $A$ an object of $\bbT$, the functor $\ev(A)$ maps a model to the interpretation of $A$ in it; a functor of the form $\ev(A) : \Mod(\bbT)\to\Set$ is called \emph{definable}.

The crucial point is to note that any pretopos functor $F : \Set^S\to\Set$, with $S$ a set, gives by postcomposition an $S$-ary operation on $\Mod(\bbT)$ that we write $F_* : \Mod(\bbT)^S\to\Mod(\bbT)$; and definable functors $\ev(A) : \Mod(\bbT)\to\Set$ always preserve these operations, that is, they map $F_*$ to $F$.
Thus, we could have taken \emph{all} such $F\in\Pretop(\Set^S,\Set)$ as $S$-ary operations on $\Mod(\bbT)$ for the structure on $\Mod(\bbT)$ needed for conceptual completeness.

But $\Pretop(\Set^S,\Set)$ is hard to work with, and this is where ultraproducts come into play: a result stated by Joyal \cite{Joyal}, and revisited by Blass \cite{Blass2}, ensures that any such functor $F$ is canonically isomorphic to a filtered (large) colimit of ultraproduct functors. We can therefore restrict our attention to the ultraproduct functors $\int_{\mu} : \Set^S \to \Set$, for $\mu$ an ultrafilter on $S$.

This motivates the notion of \emph{ultracategory}: a category $\C$ with abstract ultraproduct operations $\int_{\mu} : \C^S\to\C$ for any ultrafilter $\mu$ on $S$. The prototypical examples of ultracategories are given by the categories $\Mod(\bbT)$ and, in particular, the category $\Set$.
Ultracategories form a 2-category $\Ult$ and Makkai proved that the functor $\ev$ from above corestricts to an equivalence of categories
\[\ev : \bbT\longrightarrow\Ult(\Mod(\bbT),\Set),\]
showing how to reconstruct $\bbT$ from its ultracategory of models.
He deduces then that the functor $\Mod : \Pretop^{\op} \to \Ult$ is fully faithful, and, as the forgetful functor from ultracategories to categories is conservative, the functor $\Mod : \Pretop^{\op} \rightarrow \CAT$ is conservative. This means that a functor of pretopoi is an equivalence if and only if it induces an equivalence on its categories of models. Restricted to Boolean pretopoi, this result reduces to the usual Gödel completeness theorem for classical first-order logic.

\begin{nrmk}
    Makkai calls the first result ($\Mod : \Pretop^{\op} \to \Ult$ fully faithful) \emph{strong conceptual completeness}, and the second result ($\Mod : \Pretop^{\op} \hookrightarrow \CAT$ conservative) \emph{conceptual completeness}.
    In the above case, strong conceptual completeness implies conceptual completeness only because $\Ult$ is conservative over $\CAT$. However, this will not be the case for the future extensions of this duality, thus strong conceptual completeness does not always imply conceptual completeness. We will never use the terminology ‘‘strong conceptual completeness’’ as we find it confusing, and use \emph{conceptual completeness} interchangeably with \emph{reconstruction theorem} to express that the functor $\Mod$ mapping a theory to its category of models with some suitable structure is fully faithful.
\end{nrmk}

\begin{nrmk}
    Owing to Joyal's result about $\Pretop(\Set^S,\Set)$ one could have asked for ultracategories to have filtered colimits and for ultrafunctors to preserve them (as is done in \cite[Definition 3.1.1]{Ivan}). It is a surprising fact that the preservation of filtered colimits is not required for conceptual completeness; the above result shows that ultrafunctors between ultracategories of the form $\Mod(\bbT)$ always preserve filtered colimits. This remark should be compared with \cite[Proposition 5.3.4]{Lurie}.
\end{nrmk}

\begin{nrmk}
    We have explained how ultraproducts arise from purely categorical considerations, by looking at the category of pretopos endofunctors on $\Set$. From the point of view of a model theorist, the appearance of ultraproducts in the conceptual completeness is not so surprising: ultraproducts play a central role in the model theory of first-order logic; we refer to \cite{Kleisler} for a general survey of their usefulness. We note in particular that the ultrafilter principle and completeness of first-order logic are mutually equivalent in ZF (and slightly weaker than choice). 
\end{nrmk}

Makkai presents his result as a categorification of the \emph{Stone duality}. Let us take a look at the posetal case of the above: a posetal pretopos (or a propositional coherent theory) is a distributive lattice $A$, and its models are given by $\Mod(A) = \Latt(A,2)$.
As above, any lattice homomorphism $\mu : 2^S\to 2$ induces an operation $\mu_* : \Mod(A)^S\to\Mod(A)$. In fact, such a lattice homomorphism corresponds to an ultrafilter on $S$, and the induced operation can be seen as taking the limit of a family of $S$-many points along this ultrafilter. Hence, the operations $\mu_*$ endow the set $\Mod(A)$ with a notion of convergence along ultrafilters, which corresponds to a $\beta$-algebra structure on $\Mod(A)$ for $\beta$ the usual \emph{ultrafilter monad} on $\Set$. It is a well-known result that $\beta$-algebras are exactly compact Hausdorff spaces; hence these operations yield a topology on $\Mod(A)$ which recovers the usual logical topology (the Stone topology if $A$ is Boolean). For $A$ Boolean, Makkai's result shows that $A\simeq\TopSp(\Mod(A),2)$, thus recovering the usual Stone duality: $A$ is isomorphic to the lattice of clopens of $\Mod(A)$.

This suggests that the ultrastructure on a category of models of a coherent theory can be thought of as a ‘‘categorified Stone topology’’, ultraproducts of models as ‘‘categorified ultraconvergence’’, and ultracategories as ‘‘categorified $\beta$-algebras’’.
Following this idea, there is a pseudomonad $\bb : \CAT \to \CAT$ on locally small categories categorifying the usual ultrafilter monad $\beta$, and ultracategories are pseudoalgebras for this pseudomonad, which we call the \emph{ultrafamilies pseudomonad}. This $\bb$ has already been defined, under different names, by Rosolini--Garner and Hamad in not yet published works \cite{Hamad,Pino}.

\subsection{Lurie's extension}
Extending on this work of Makkai, Lurie gives in \cite{Lurie} a new proof of Makkai's results.
The main difference in Lurie's approach is the use of \emph{coherent topoi} instead of small pretopoi to represent coherent theories. Let us consider the sheaf functor $\Sh:\Pretop^{\op} \to \CohTop$ sending a pretopos to its classifying topos, given by the sheaves over it for the coherent topology. 
This functor is faithful, full on equivalences, but not full: the maps (geometric morphisms) between coherent topoi are geometric interpretations between coherent theories, and such an interpretation is in the image of the faithful functor $\Sh$ above if and only if it does not use the full power of geometric logic but only its finitary part (\ie coherent logic).
Essentially, going from pretopoi to coherent topoi amounts to allowing more interpretations between the coherent theories, and accordingly, Lurie introduces a more flexible notion of functor between ultracategories: \emph{left-ultrafunctors}.

The main theorem of Lurie can then be stated as follows: for a coherent topos $\E$ with ultracategory of points $\pt(\E)$, the functor
\[\E\longrightarrow\UltL(\pt(\E),\Set)\]
(where $\UltL$ denotes the 2-category of ultracategories and left-ultrafunctors) is an equivalence, and this equivalence can be restricted to recover Makkai's result.
The situation can be summarized by the following diagram of 2-categories:
\[\begin{tikzcd}
	{\Pretop^{\op}} & \Ult \\
	\CohTop & {\UltL}
	\arrow[hook, "\text{full}", from=1-1, to=1-2]
	\arrow[hook, from=1-1, to=2-1]
	\arrow[hook, from=1-2, to=2-2]
	\arrow[hook, "\text{full}", from=2-1, to=2-2]
\end{tikzcd}\]
where the two horizontal functors are fully faithful, whereas the vertical ones are only faithful (and induce equivalences of the core groupoids). Thus, the notion of left-ultrafunctor is the one to consider when seeing ultracategories as categorified topological spaces.

\begin{nrmk}
    As explained above, we cannot deduce from Lurie's result that the functor $\Mod : \CohTop \to \CAT$ is conservative; being a left-ultrafunctor is a structure on a functor, and so the forgetful functor $\UltL\to\CAT$ is not conservative.
\end{nrmk}

\subsection{Virtual ultracategories}
Virtual ultracategories arise naturally from the following observation: the ultraproduct of models is a categorification of the convergence of ultrafilters; but all topologies on a set can be recovered from their ultraconvergence relations, not only compact Hausdorff ones; hence, we should be able to generalize the ultraproduct operation to all topoi.

The ultraconvergence structure on a topological space is no longer given by an operation as for the compact Hausdorff case, but by a \emph{relation} between ultrafilters and points. The objects resulting from the axiomatization of topological spaces through the ultraconvergence relation are called \emph{relational $\beta$-modules} and are described in \cite{Barr}. Virtual ultracategories categorify relational $\beta$-modules, replacing relations by distributors (also called profunctors) and the monad $\beta$ by the pseudomonad $\bb$, so another name for virtual ultracategories could have been \emph{distributional $\bb$-modules}.

As explained in \cite[Example 4.13]{GenMultiCat}, relational $\beta$-modules can be seen as \emph{generalized multicategories}: objects are given by the points of the space, and there is a unique generalized arrow from a point $a$ to an ultrafamily of points $(b_s)_{s:\mu}$ if the image of $(b_s)_{s:\mu}$ converges to $a$. We get the notion of virtual ultracategory by removing the posetal restriction: there can be several arrows from a point $a$ to an ultrafamily of points $(b_s)_{s:\mu}$.
In the same way, we expect virtual ultracategories to be a generalized multicategorical notion. Concretely, a virtual ultracategory has generalized arrows of the shape $a\ultrato(b_s)_{s:\mu}$ with domain a single object $a$ and codomain an ultrafamily of objects $(b_s)_{s:\mu}$.

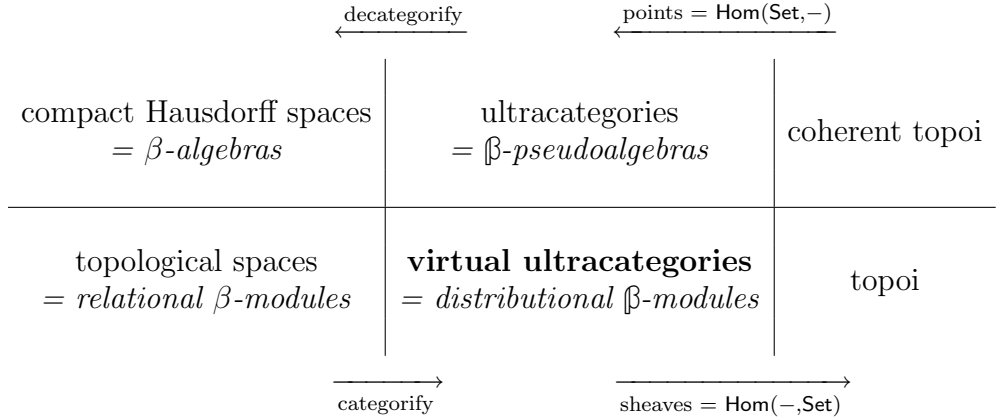
\vspace{9pt} 
\begin{center}
\begin{tabular}{l  r}
    \hspace{60pt}
    $\xleftarrow[]{\text{decategorify}}$
    & \hspace{40pt} $\xleftarrow[]{\text{points = } \Hom(\Set,-)}$
\end{tabular}
\nopagebreak

\begin{tabular}{ c | c | c }
 & & \\
\begin{tabular}{@{}c@{}}compact Hausdorff spaces \\ \emph{= $\beta$-algebras}\end{tabular}
    & \begin{tabular}{@{}c@{}} ultracategories \\ \emph{= $\bb$-pseudoalgebras}\end{tabular}
        &  {coherent topoi}
\\ & & \\\hline
 & & \\
\begin{tabular}{@{}c@{}}topological spaces \\ \emph{= relational $\beta$-modules}\end{tabular}
    & \begin{tabular}{@{}c@{}} \textbf{virtual ultracategories} \\ \emph{= distributional $\bb$-modules}\end{tabular}
        &  topoi
\\ & &
\end{tabular}
\nopagebreak

\begin{tabular}{ l  r}
    \hspace{60pt}
    $\xrightarrow[\text{categorify}]{}$
    & \hspace{50pt} 
    $\xrightarrow[\text{sheaves = } \Hom(-,\Set)]{}$
\end{tabular} 
\end{center}
\vspace{9pt}

\begin{nrmk}
    Virtual ultracategories are to ultracategories what multicategories are to monoidal categories: a generalized arrow of the shape $a\ultrato(b_s)_{s:\mu}$ can be thought of as an arrow from $a$ into the potential ultraproduct $\int_{s:\mu}b_s$ of the $(b_s)_{s:\mu}$. We expect that virtual ultracategories generalize ultracategories in the precise sense that ultracategories are the virtual ultracategories satisfying the representability condition given in \cite[Section 9]{GenMultiCat}. This justifies the adjective ‘‘virtual’’: virtual ultracategories are to ultracategories what multicategories are to monoidal categories (or what virtual double categories are to double categories).
\end{nrmk}

\subsection{The reconstruction theorem}
Following the analogy with topological spaces, starting from a topos $\E$ we can construct a virtual ultrastructure $\vpt(\E)$ on its points, extending the usual category of points of $\E$. We might hope that $\vpt(\E)$ is enough to recover the original topos; the aim of the paper is to show that it is the case if the topos has enough points. More precisely, we show that the evaluation functor
\[\E\longrightarrow\vUlt(\vpt(\E),\Set)\]
is an equivalence if $\E$ has enough points, and that we recover Lurie's theorem \cite[Theorem 0.0.6]{Lurie} in the case when $\E$ is coherent.

We will see this result as a \emph{reconstruction result} for topoi with enough points, showing that one can recover a topos with enough points (seen as a complete geometric theory) from the virtual ultracategory of its points (the models of the theory together with some ‘‘topological structure’’ on them). We will present it as a pseudoidempotent 2-adjunction with $\Set$ playing the role of a dualizing object.
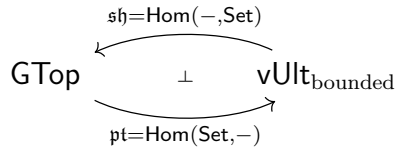
\[\begin{tikzcd}
    \GTop && \vUltb
    \arrow[""{name=0, anchor=center, inner sep=0}, "{\vpt = \Hom(\Set,-)}"', curve={height=18pt}, from=1-1, to=1-3]
    \arrow[""{name=1, anchor=center, inner sep=0}, "{\vSh = \Hom(-,\Set)}"', curve={height=18pt}, from=1-3, to=1-1]
    \arrow["\vdash"{anchor=center, rotate=90}, draw=none, from=0, to=1]
\end{tikzcd}\]
The induced comonad on $\GTop$ corresponds to the \emph{restriction of a topos to all its points}, and thus this adjunction induces a reflection of Grothendieck topoi with enough points into bounded virtual ultracategories.

\subsection*{Structure of the paper}
In Section \ref{sec:Ultracombi}, we recall basic notions and facts about ultrafilters; in particular, we define the functor fat beta $\bb : \CAT\to\CAT$ categorifying the usual $\beta:\Set\to\Set$. Sections \ref{sec:ultracvg} and \ref{sec:ultracat} elaborate on the first sections in two different directions: in Section \ref{sec:ultracvg} we recall how one can see topological spaces as relational $\beta$-modules and we prove a characterization of étale maps through ultraconvergence (Theorem \ref{Thm:étale-ucvg}) that is essentially the 0-dimensional case of our reconstruction theorem; in Section \ref{sec:ultracat} we fix the definition of ultracategories we work with. Sections \ref{sec:vultracat} and \ref{sec:back-forth} are dedicated to, respectively, introducing virtual ultracategories and constructing functors relating them to topoi; in particular, this allows us to state our reconstruction theorem. Sections \ref{sec:desc} and \ref{sec:ample} constitute the core of the proof: in Section \ref{sec:desc} we prove a crucial result giving a sufficient condition for a functor of virtual ultracategories to be effective descent (Proposition \ref{prop:desc-vUlt}); in Section \ref{sec:ample} we show how to build a specific topological groupoid so that we can apply the descent result of Section \ref{sec:desc}. Section \ref{sec:ccl} acts as a conclusion: we prove the reconstruction theorem (Theorem \ref{Thm:thetrueone}) and present it as an idempotent adjunction (Corollary \ref{cor:final-adj}). 

\subsection*{Conventions}
We work in a classical metatheory with choice and one universe. The universe gives us a notion of small sets; when not specified, a set is small by default. We use the terminology large set or class to designate a, not necessarily small, set. We denote by $\Set$ the category of small sets; by $\Cat$ the category of small categories; and by $\CAT$ the category of large locally small categories.

When working with strict 2-categories (\ie categories enriched in categories), we use the prefix ``2-'' for the strict notion, and the prefix ``pseudo-'' for the weak notion. For example, a 2-pullback is defined up to strict isomorphism, while a pseudopullback is only defined up to internal equivalence.

We use the abbreviation \emph{lex} for left-exact: a lex category is a category with finite limits, and a lex functor is a finitely continuous functor.
A \emph{pretopos} (resp. \emph{infinitary pretopos}) is a category satisfying the following exactness properties:
\begin{enumerate}
    \item it is lex,
    \item it has a strict initial object,
    \item it has finite (resp. small) disjoint coproducts that are pullback stable, and
    \item it has quotient of equivalence relations that are pullback stable. 
\end{enumerate}
Morphisms of pretopoi are functors preserving the pretopos structure; we denote by $\Pretop$ the 2-category of small pretopoi, pretopos morphisms, and natural transformations between them.

By \emph{topos} we will always mean Grothendieck topos, that is (by Giraud's theorem) an accessible infinitary pretopos. Topoi form a strict 2-category $\GTop$: 1-arrows are \emph{geometric morphisms}, a geometric morphism from $\E$ to $\F$ is given by its \emph{inverse image} that is a lex and cocontinuous functor from $\F$ to $\E$; 2-arrows are natural transformations between the inverse images.
We denote by $\SetO$ the topos classifying the theory of objects, $\SetO$ represents the forgetful functor from topoi to categories, that is, $\E$ and $\GTop(\E,\SetO)$ are equivalent categories.

We denote by $\TopSp$ the locally posetal 2-category of topological spaces: 1-arrows are continuous maps, and 2-arrows are given by the specialization order: for $f,g:T\to T'$ continuous maps, $f\cleq g$ if and only if, for all $x\in T$ and $U$ open of $T'$, $U\ni f(x)\implies U\ni g(x)$. In particular, taking sheaves induces a strict 2-functor $\TopSp\to\GTop$.

We denote by $\Sh(T)$ the category of sheaves over a topological space $T$; we will freely identify a sheaf over $T$ with its étale space $E\to T$.

\section{Ultracombinatorics}\label{sec:Ultracombi}

In this section, we recall some basic definitions and constructions with ultrafilters.

\begin{defn}
    An \emph{ultrafilter} on a set $S$ is a lattice homomorphism $\mu : 2^S\to 2$. A subset $A\subseteq S$ such that $\mu(A)=1$ is said \emph{$\mu$-large}. We can rephrase the lattice homomorphism axioms in terms of $\mu$-large subsets:
    \begin{enumerate}
        \item $\emptyset$ is not $\mu$-large,
        \item $\mu$-large subsets are upward closed for $\subseteq$,
        \item $\mu$-large subsets are closed under finite intersections (in particular $S$ is $\mu$-large),
        \item if $A\cup B$ is $\mu$-large, then at least one of $A$ or $B$ is $\mu$-large. 
    \end{enumerate}
    Another way to see an ultrafilter is as a finitely additive $\{0,1\}$-valued measure on $S$, and following \cite{Lurie} we adopt the notation $\mu,\nu,...$ for ultrafilters.
\end{defn}

\begin{ex}
    For each $s$ in $S$ there is an ultrafilter $\delta_s$ called the \emph{principal ultrafilter at $s$} denoting the measure concentrated at $s$: a subset is $\delta_s$-large if and only if it contains $s$. This gives a map $\delta : S\hookrightarrow\beta(S)$.
\end{ex}

\begin{defn}
    The \emph{Stone--\v{C}ech compactification} of a set $S$ is the topological space $\beta(S)$ whose points are ultrafilters on $S$ and whose topology is generated by the basic opens
    \[\{\mu\in\beta(S) : A\text{ is $\mu$-large}\}\]
    with $A$ ranging over the subsets of $S$. 
\end{defn}
The following proposition is standard.
\begin{prop}\label{prop:betaUP}
    The above basis defines a compact Hausdorff topology on $\beta(S)$ satisfying the following universal property: for $K$ compact Hausdorff, any map of sets $S\to K$ can be uniquely extended along $\delta$ into a continuous map $\beta(S)\to K$, \ie $\beta(S)$ is the free compact Hausdorff space on the discrete set of points $S$.
\end{prop}

\begin{defn}
    Let $f : S\to T$ be a map of sets. We define $f_* : \beta(S)\to\beta(T)$ to be the extension of $S\to \beta(T), s\mapsto\delta_{f(s)}$ given by Proposition \ref{prop:betaUP}, we will call $f_*(\mu)$ the \emph{pushforward of $\mu$ along $f$}.
\end{defn}

One can compute $f_*(\mu)$, a subset $B\subseteq T$ is $f_*(\mu)$-large if and only if its preimage is $\mu$-large. Note also that pushing forward along $f$ corresponds to precomposing with $2^T\to 2^S$.

\begin{rmk}
    Hence $\beta$ induces a functor from $\Set$ to the category of compact Hausdorff spaces, and the universal property above shows that it is actually the left adjoint to the forgetful functor, giving the underlying set of a space. Actually, this adjunction is monadic: the functor $\beta : \Set \to \Set$ has a monad structure whose algebras are compact Hausdorff spaces.
\end{rmk}

\begin{nota}
    For $K$ compact Hausdorff and $(a_s)_{s:S}$ an $S$-family of points in $K$, \ie a map of sets $a : S\to K$, we denote by $\int_{s:\mu}a_s$ the image of $\mu$ under the extension $\beta(S)\to K$ given by Proposition \ref{prop:betaUP}. We will think of $\int_{s:\mu}a_s$ as the limit of the family $(a_s)_{s:S}$ directed by $\mu$.
    For example, $f_*(\mu) = \int_{s:\mu}\delta_{f(s)}$. More generally, given $(\nu_s)_{s:S}$ an $S$-family of ultrafilters on $T$ and $\mu$ an ultrafilter on $S$, we can consider $\nu\coloneqq\int_{s:\mu}\nu_s\in\beta(T)$ the limit of $(\nu_s)_{s:S}$ directed by $\mu$: a subset is $\nu$-large if and only if it is $\nu_s$-large for a $\mu$-large number of $s$. 
\end{nota}

We interpret an ultrafilter on $S$ as encoding a way to take a limit of an $S$-indexed family of points in a (compact Hausdorff) space. For example, the principal ultrafilter $\delta_{x_0}$ encodes the trivial convergence to $x_0$: we always have $\int_{x:\delta_{x_0}}a_x = a_{x_0}$.
If $S_0\subseteq S$ is $\mu$-large, then $\mu$ restricts to an ultrafilter $\mu_{\res{S_0}}$ on $S_0$, and taking the limit of the restricted family along $\mu_{\res{S_0}}$ gives the same result, \ie $\int_{s:\mu}a_s = \int_{s:\mu_{\res{S_0}}}a_s$. Hence $\mu$ and $\mu_{\res{S_0}}$ encode the same way to take a limit, we will say that they have the same \emph{type of ultraconvergence}. This motivates the following definition.

\begin{defn}\label{def:iso-uf}
    An \emph{isomorphism} between two ultrafilters $\mu\in\beta(S)$ and $\nu\in\beta(T)$ is given by a bijection between a $\mu$-large subset $S_0$ and a $\nu$-large subset $T_0$ identifying the restrictions $\mu_{\res{S_0}}$ and $\nu_{\res{T_0}}$.
    Two ultrafilters are said to have the same \emph{type of ultraconvergence} if there exists an isomorphism between them.
\end{defn}

\begin{rmk}
    Note that $\mu$ and $\nu$ are isomorphic if and only if $\mu\in\beta(S)$ and $\nu\in\beta(T)$ have homeomorphic neighborhoods. 
\end{rmk}

The underlying set of an ultrafilter is not invariant under isomorphism of ultrafilters. The main tool to manipulate ultrafilters without specifying the underlying set is the following notion of \emph{ultrafamily}.

\begin{defn} 
    Let $\mu$ be an ultrafilter on $S$ and $A$ a set (or a class). A \emph{$\mu$-family} $(a_s)_{s:\mu}$ of elements of $A$ is a family $(a_s)_{s:S_0}$ defined on a $\mu$-large set $S_0$ modulo the relation identifying two such families if they coincide on a $\mu$-large set. A $\mu$-family will be denoted by $(a_s)_{s:\mu}$ and we will denote by $A^{\mu}$ the set of $\mu$-families of elements of $A$. 
    One can check that $A^{\mu} = \colim_{\mu(S_0) = 1}(A^{S_0})$ and that $A^{\mu}$ does not depend on the set underlying $\mu$.
\end{defn}

We now generalize the notion of ultrafamily to the one of \emph{dependent ultrafamily}, this recovers the usual notion of ultraproduct of a family of sets.

\begin{defn}
    For $(A_s)_{s:\mu}$ an ultrafamily of sets (or classes), a \emph{dependent ultrafamily} is given by a dependent family $(a_s \in A_s)_{s:S_0}$ defined on a $\mu$-large set $S_0$, modulo the relation identifying two families if they coincide on a $\mu$-large set, we will denote it by $(a_s\in A_s)_{s:\mu}$ or by $a : (s:\mu) \to A_s$. We denote by $\int_{s:\mu}A_s$ the set of such $\mu$-dependent families, it is called the \emph{ultraproduct} of the $(A_s)_{s:\mu}$. We recover the above notion $A^{\mu}$ by taking $\int_{\mu}A$, the ultraproduct of the constant family.
\end{defn}

\begin{lem}\label{lem:up-prod}
    Ultraproducts commute with (dependent) products of sets. For $(A_s)_{s:\mu}$ a $\mu$-family of sets, and $(B_s(a))_{a\in A_s}$ a dependent family of sets over $A_s$, there is a canonical bijection
    \[\int_{s:\mu}(a\in A_s)\times B_s(a)\stackrel{\sim}\longrightarrow\left((a_s) \in \int_{s:\mu}A_s\right)\times\int_{s:\mu}B_s(a_s),\]
    mapping $(a_s,b_s)_{s:\mu}$ onto the dependent pair $((a_s\in A_s)_{s:\mu},(b_s\in B_s(a_s))_{s:\mu})$.
\end{lem}
The proof is straightforward. Note, however, that ultraproducts do not commute with infinite products.

\begin{defn}
    Ultraproducts reduce to a generalized quantifier in the propositional world, thus defining a notion of \emph{ultraquantification} for an ultrafamily of propositions. 
    Let $(\varphi(s))_{s\in S}$ be an $S$-family of subsingletons, \ie $\varphi(s)$ is a subset of $1$ and is thought of as a proposition. Then the ultraproduct $\int_{s:\mu}\varphi(s) \subseteq 1^{\mu} = 1$ is also a subsingleton, and is inhabited if and only if there is a $\mu$-large set on which all $\varphi(s)$ are inhabited. We will write $(\forall s:\mu) \varphi(s)$ instead of $\int_{s:\mu}\varphi(s)$, and say that the $\varphi(s)$ are \emph{true for $s$ in $\mu$} or \emph{true for $\mu$-all $s$}.
\end{defn}

\begin{prop}\label{prop:ultraquantification}\leavevmode
    \begin{enumerate}[label=(\roman*)]
        \item\label{autodual} Ultraquantification is autodual:
        \[(\forall s:\mu)\neg \varphi(s)\Longleftrightarrow\neg(\forall s:\mu)\varphi(s).\]
        \item\label{prop:ultraquantification2} For $(A_s)$ a $\mu$-family of sets and $(\psi_s)$ predicates on the $(A_s)$,
        \[(\forall s:\mu)(\forall a\in A_s)\psi_s(a)\Longleftrightarrow (\forall (a_s)_{s:\mu}\in\textstyle\int_{s:\mu}A_s)(\forall s:\mu)\psi_s(a_s).\]
    \end{enumerate}
\end{prop}
\begin{proof}\leavevmode
    \begin{enumerate}
        \item Either the set of $s$ satisfying $\varphi(s)$ or its complement is $\mu$-large.
        \item By negating both sides of the equivalence and applying part~\ref{autodual}, this amounts to the following:
        \[(\forall s:\mu)(\exists a\in A_s)\neg \psi_s(a) \Longleftrightarrow (\exists (a_s)_{s:\mu}\in\textstyle\int_{s:\mu}A_s)(\forall s:\mu)\neg \psi_s(a_s),\]
        and this follows from Lemma \ref{lem:up-prod}. \qedhere
    \end{enumerate}
\end{proof}

When we work with categories instead of sets, the colimit defining ultraproducts should be taken 2-categorically; this is the purpose of the next definition.
\begin{defn}
    Let $\C$ be a category, we denote by $\C^{\mu}$ the filtered 2-colimit of $(\C^{S_0})_{\mu(S_0)=1}$ taken in the 2-category $\Cat$. Explicitly, its objects are $S$-families of objects of $\C$ and the set of arrows from $(c_s)_{s:S}$ to $(d_s)_{s:S}$ is given by $\int_{s:\mu}\C(c_s,d_s)$.
\end{defn}
\begin{rmk}
    Whereas $A^{\mu}$ does not depend on the set underlying $\mu$ up to bijection of sets, the category $\C^{\mu}$ does not depend on the set underlying $\mu$ only up to equivalence. So, if $\C$ is discrete, $\C^{\mu}$ can be identified with the ultraproduct of sets from above only up to equivalence.
\end{rmk}

Ultrafamilies on a fixed set $X$ form a category.

\begin{defn}
    Let $X$ be a set (or a class), the category of \emph{ultrafamilies in $X$} is defined as follows:
    its objects are pairs of an ultrafilter $\mu$ on some set $S$ together with a family $(x_s) \in X^S$; and arrows from $(\mu,(x_s)_{s:S})$ to $(\nu,(y_t)_{t:T})$ are given by $f\in T^{\mu}$ such that $\int_{s:\mu}f(s) = \nu$ and $(y(f(s)))_{s:\mu} = (x(s))_{s:\mu}$. 
    \[\begin{tikzcd}
    	\mu && \nu \\
    	& X
    	\arrow["f", maps to, from=1-1, to=1-3]
    	\arrow["x"', dashed, from=1-1, to=2-2]
    	\arrow["y", dashed, from=1-3, to=2-2]
    \end{tikzcd}\]
    Up to isomorphisms, its objects are ultrafamilies $(x_s)_{s:\mu}$ in $X$. We denote this category by $\UF_X$ (it corresponds to $\UF_{\Set^X}$ of \cite[Definition 21]{GarnerUF}).
\end{defn}

\begin{defn}
    For a set (or a class) $X$, we define $\bb(X)\coloneqq\UF_X^{\op}$. This $\bb$ is functorial in $X$ and the functor $\Set\to\CAT$ obtained is called \emph{fat beta}.  
\end{defn}

\begin{rmk}\label{rmk:largebb}
    We can extend this construction to any locally small category $\C$. Note that $\bb(\C)$ is the opposite of the category $\UF_{\C^{\op}}$ (and not of $\UF_{\C}$). The construction $\UF_{\C}$ can be thought of as a variation of the usual ‘‘free sum completion’’ $\mathsf{Fam}(\C)$.
\end{rmk}

\begin{nota}
    We write $\UF$ instead of $\UF_1$, this recovers the usual category of ultrafilters, first introduced by Blass in \cite{Blass}. Explicitly, for $\mu$ and $\nu$ two ultrafilters on, respectively, $S$ and $T$, a map of ultrafilters $f\in\UF(\mu,\nu)$ is given by a $\mu$-family $(f(s))_{s:\mu}$ in $T$ such that $\int_{s:\mu}f(s) = \nu$.  
    Note that isomorphisms of $\UF$ recover the notion of Definition \ref{def:iso-uf}.
\end{nota}

\begin{rmk}
    For any category $\C$, the mapping $\mu \mapsto \C^{\mu}$ forms a strict 2-functor $\UF^{\op} \to \Cat$.
    Moreover, the category of elements of this functor recovers $\bb(\C)$.
\end{rmk}

We finish this section by defining \emph{the dependent sum operation} on ultrafilters.

\begin{defn}
    Let $(T_s)_{s:S}$ be an $S$-family of sets, $\mu$ an ultrafilter on $S$, and $(\nu_s)_{s:S}$ ultrafilters on the $(T_s)_{s:S}$. The \emph{sum ultrafilter} $\sum_{s:\mu}\nu_s$ is the ultrafilter $\int_{s:\mu}(i_s)_*(\nu_s)$ on $T\coloneqq\sum_{s:S}T_s$, where $i_s : T_s \hookrightarrow T$ is the canonical inclusion. Concretely $\sum_{s:S}A_s\subseteq \sum_{s:S}T_s$ is $\sum_{s:\mu}\nu_s$-large if $A_s$ is $\nu_s$-large for $\mu$-all $s$.
    Note that if the family of sets $(T_s)$ is constant to a set $T$ we recover $\int_{s:\mu}\nu_s$ as $\pi_*(\sum_{s:\mu}\nu_s)$, where $\pi$ denotes the second projection map $\sum_{s:S}T \to T$.
    
    We sometimes write $(s:\mu)\otimes\nu_s$ instead of $\sum_{s:\mu}\nu_s$, and if the family $(\nu_s)$ is constant to $\nu$, we denote the sum ultrafilter by $\mu\otimes\nu$.

    The sum of ultrafilters is functorial: for $\mu$ an ultrafilter, the sum operation gives a functor $\sum_{\mu} : \UF^{\mu}\to\UF$; we have more, the sum operation induces a functor $\UF_{\UF}\to\UF$ (this is the opposite of the functor $\bb(\bb(1))\to\bb(1)$ showing that $\bb(1)$ is an ultracategory, see Example \ref{ex:bb-ultracat}).
\end{defn}

\begin{rmk}
    In general $\mu\otimes\nu$ and $\nu\otimes\mu$ are not comparable. They are both extensions of the filter $\mu\times\nu$ generated by the products of a $\mu$-large set with a $\nu$-large set, but in full generality we cannot say more.
\end{rmk}

\begin{cons}\label{cons:ass-unit-sum}
     We denote by $\ptuf$ the unique ultrafilter on the singleton set, so all principal ultrafilters have the same type of ultraconvergence as $\ptuf$. This plays the role of a neutral for the sum operation: we have the following isomorphism of ultrafilters:
    \begin{equation*}
        \mu\otimes\ptuf \simeq \mu \hspace{30pt} \ptuf\otimes\mu \simeq \mu \hspace{30pt} \sum_{r:\lambda}\sum_{s:\mu_r}\nu_s \simeq \sum_{s:\sum_{r:\lambda}\mu_r}\nu_s.
    \end{equation*}
\end{cons} 

\begin{cons}\label{cons:curry}
    For a category $\C$, we have a canonical curryfication isomorphism:
    \[\C^{\sum_{s:\mu}\nu_s}\stackrel{\sim}\longrightarrow \int_{s:\mu}\C^{\nu_s}\]
    that we can  write in a more suggestive way
    \[((s:\mu\otimes \nu_s)\to\C) \cong (s:\mu\to(\nu_s\to\C)).\] 
\end{cons}

\section{Ultraconvergence in topological spaces}\label{sec:ultracvg}

Let $\mu$ be an ultrafilter on a set $T$. We think of $\mu$ as a $\mu$-family of points $(x)_{x:\mu}$ on $T$. Notice that if $A\subseteq T$ then $A$ is $\mu$-large can be expressed as $\forall x:\mu, x\in A$, or also as $\mu\in\beta(A)\subseteq\beta(T)$. Hence, we will sometimes say that $A$ contains $\mu$, and write $\mu\in A$, to express that $A$ is $\mu$-large.
If now the set $T$ is endowed with a topology, we can express the fact that a point is a limit point of an ultrafilter (seen as an ultrafamily of points as explained above); this should be compared with the usual sequential convergence of points.

\begin{defn}
    Let $T$ be a topological space, $a$ a point of $T$, and $\mu$ an ultrafilter on the set of points of $T$. We say that \emph{the ultrafilter $\mu$ converges to the point $a$}, and we write $a\cleq\mu$, if any open containing $a$ contains $\mu$, \ie
    \[(\forall U\ \text{open})\ U\ni a \implies U\ni \mu.\]
    If $(x_s)_{s:\mu}$ is a $\mu$-family of points of the space we write $a\cleq (x_s)_{s:\mu}$ instead of $a\cleq x_*(\mu)$, \ie
    \[(\forall U\ \text{open})\ U\ni a \implies (\forall x:\mu)\ U\ni x.\]
\end{defn}

\begin{ex}
    If $T$ is compact Hausdorff, $a\cleq\mu$ if and only if $\int_{x:\mu}x = a$.
\end{ex}

\begin{ex}
    A sequence $(x_n)_{n:\bbN}$ converges to a point $a$ if and only if $a\cleq(x_n)_{n:\mu}$ for all ultrafilters $\mu$ non-principal on $\bbN$.
\end{ex}

\begin{rmk}
    The ultraconvergence relation generalizes the specialization order on points:  $a\cleq \delta_b$ if and only if $a$ is less than $b$ for the specialization order. (Recall that the specialization order on the points of a space is given by: $a$ is less than $b$ if and only if any open containing $a$ also contains $b$ or equivalently, $a$ belongs to the closure of $\{b\}$.)
    Ultraconvergence can thus be seen as an extension of the specialization order that is (as shown by the following proposition) strong enough to recover the topology of the space.
\end{rmk}

\begin{prop}\label{charact:open}
    A subset $A\subseteq T$ is open if and only if
    \[(\forall a\cleq\mu)\ a\in A\implies \mu\in A.\]
\end{prop}
\begin{proof}
    Suppose $a\in A$ open and $\mu$ converges to $a$, then by definition of the ultraconvergence $A$ contains $\mu$.
    Conversely, let $a\in A$ and consider $\eF_a$ the filter of neighborhoods of $a$, note that $T\setminus A$ is never in $\eF_a$. So, if $A$ were not in $\eF_a$, we could have extended $\eF_a\cup\{T\setminus A\}$ to an ultrafilter converging to $a$, contradicting the hypothesis. Hence $A$ is in $\eF_a$, showing that $A$ is a neighborhood of $a$.
\end{proof}

\begin{rmk}
    In a similar way, one can show the following:
    \begin{enumerate}
        \item $A$ is closed if and only if
        $(\forall a\cleq\mu)\ \mu\in A\implies a\in A$.
        \item The interior of $A$ is given by
        $\{a\in T : (\forall \mu\cgeq a) \mu \in A\}$.
        \item The closure of $A$ is given by
        $\{a\in T : (\exists \mu\cgeq a) \mu \in A\}$.
        \item The set of points that are limits of an ultrafilter $\mu$ is given by
        $\bigcap_{\mu(A)=1}\cl(A)$, the intersection of all closures of $A$ for $A$ a $\mu$-large subset.
        \item A topological space is discrete if and only if $a\cleq\mu$ only for $\mu = \delta_a$.
        \item A topological space is codiscrete if and only if any ultrafilter converges to any point.
    \end{enumerate}
\end{rmk}

Hence, one can define a topological space using the ultraconvergence relation instead of the lattice of open subsets. This is the approach taken by Barr to give a purely algebraic description of topological spaces as \emph{relational} algebras: topological spaces can be defined as suitable algebras for the $\beta$ monad extended to the category of sets and relations. We only give a concrete definition of relational $\beta$-modules and refer to \cite{Barr} for a more detailed exposition of the algebraic side of the story.

\begin{defn}\label{def:rel-beta-mod}
    A \emph{relational $\beta$-module} is given by a small set $X$ together with a relation $\cleq$ between $X$ and $\beta(X)$ such that:
    \begin{enumerate}
        \item $a\cleq \delta_a$ for any $a\in X$, and 
        \item if $a\cleq \mu$, and $(\nu_x)_{x:\mu}\in\beta(X)^{\mu}$ is a $\mu$-family of ultrafilters on $X$ with $(\forall x:\mu)\ x\cleq\nu_x$, then $a\cleq\int_{x:\mu}\nu_x$.
    \end{enumerate}
     
\end{defn}

\begin{rmk}
    For any topological space $T$, the ultraconvergence relation on it gives a relational $\beta$-module; conversely, from a relational $\beta$-module one can define a topology as in Proposition \ref{charact:open}; one can then show that these two data are equivalent. We see a relational $\beta$-module as a generalized posetal multicategory: the relation $a\cleq\mu$ is seen as a generalized arrow with domain $a$ and codomain the ultrafamily $(x)_{x:\mu}$, and the conditions of Definition \ref{def:rel-beta-mod} correspond to the identity and the composition of these generalized arrows. 
\end{rmk}

The rest of this section is dedicated to the proof of Theorem \ref{Thm:étale-ucvg}, which characterizes étale maps through ultraconvergence; we first begin with a definition.

\begin{defn}\label{def:principal-over}
    Let $p : E\to T$ be a map of sets, we denote by $E_s$ the fiber of $p$ at $s\in T$. For an ultrafilter $\nu\in\beta(E)$ the following conditions are equivalent:
    \begin{enumerate}
        \item $p$ is injective on a $\nu$-large set.
        \item $p$ induces an isomorphism of ultrafilters from $\nu$ to $p_*(\nu)$.
        \item $\nu$ can be written as $\sum_{s:\mu}\delta_{e_s}$ with $\mu\in\beta(T)$ and $(e_s)\in \int_{s:\mu}E_s$.
    \end{enumerate}
    If it is the case, we say that $\nu$ is \emph{principal over} $T$, note that being principal over the terminal set recovers the usual notion of principal ultrafilter.
    We denote by $\beta(E/T)$ the subspace of $\beta(E)$ consisting of the $\nu$ principal over $T$.
\end{defn}

\begin{rmk}
    The idea behind this definition is that $\nu$ sits ‘‘horizontally’’ above its image and does not have any ‘‘vertical content’’.
\end{rmk}

\begin{Thm}\label{Thm:étale-ucvg}
    A continuous map of topological spaces $p : E \to T$ is étale if and only if
    \begin{enumerate}[label=(\roman*)]
        \item\label{cond1} any ultrafilter $\mu\cgeq p(e)$ has a unique lift $\nu\cgeq e$, and
        \item\label{cond2} any ultrafilter $\mu\cgeq p(e)$ has a lift $\nu\cgeq e$ that is principal over $T$ (\ie $p$ is injective on a $\nu$-large set).
    \end{enumerate}
\end{Thm}

\begin{rmk}
    As shown in Proposition \ref{prop:0dimcase}, this characterization gives a description of a sheaf over a topological space as a bunch of stalks $(E_x)_{x:T}$ glued by coherent maps. This result has already been proven by Lurie for compact Hausdorff spaces in \cite[Section 3.4]{Lurie}; we generalize it to all topological spaces.
\end{rmk}

\begin{rmk}
    We compare Theorem \ref{Thm:étale-ucvg} above to the result in \cite{lochomeo}. In \cite{lochomeo}, the authors define a \emph{discrete fibration} of topological spaces to be a map satisfying only the lifting condition Theorem \ref{Thm:étale-ucvg}\ref{cond1}. They show that this condition is not sufficient to characterize étale maps and give several characterizations of étale maps making use of pullbacks of spaces; however, ultrafilters on products (and pullbacks) are not easy to work with.
\end{rmk}

We prove some preliminary lemmas.

\begin{lem}\label{charact:map}
    A set map $f : A\to B$ between topological spaces is continuous if and only if $f_*(\mu)$ converges to $f(a)$ for any $\mu$ converging to $a$.
\end{lem}
\begin{proof}
    Suppose $f$ continuous and $\mu\cgeq a$, then for any open $U$ containing $f(a)$, $f^{-1}(U)$ is an open containing $a$, so $\mu$ is in $f^{-1}(U)$, and so $f_*(\mu)$ is in $U$ as wanted.
    Conversely, take $U$ an open of $B$; if $\mu\cgeq a$ with $a\in f^{-1}(U)$, then (by hypothesis) $f_*(\mu)\cgeq f(a)$ with $f(a)\in U$ open, so $f_*(\mu)$ is in $U$ and $\mu$ is in $f^{-1}(U)$, this concludes that $f^{-1}(U)$ is open by Proposition \ref{charact:open}.
\end{proof}

\begin{lem}\label{charact:openmap}
    A continuous map $f : A\to B$ such that any ultrafilter $\mu\cgeq f(a)$ can be lifted to an ultrafilter $\nu\cgeq a$ is open.
\end{lem}
\begin{proof}
    Suppose $f$ satisfies the hypothesis, and let $U$ be an open of $A$, we use Proposition \ref{charact:open} to show that $f(U)$ is an open of $B$. Let $\mu\cgeq f(a)$ with $a\in U$, by hypothesis $\mu = f_*(\nu)$ for some $\nu\cgeq a$. As $U$ is open, $\nu\in U$, and so $\mu = f_*(\nu) \in f(U)$.
\end{proof}

\begin{rmk}
    The above lemma is actually a characterization of open maps: a continuous map $f$ is open if and only if any ultrafilter $\mu\cgeq f(a)$ can be lifted to an ultrafilter $\nu\cgeq a$.
    There is a dual characterization for proper maps: a continuous map $f$ is proper if and only if any point $a\cleq f(\nu)$ can be lifted to a point $b\cleq\nu$.
\end{rmk}

\begin{lem}\label{lem:principal-over}
    Let $p:E\to T$ be a map between sets that we see as an étale map of discrete topological spaces. The pushforward of $p$ along the continuous injection $\delta : T\hookrightarrow\beta(T)$ is given by $q\coloneqq (p_*)_{\res{\beta(E/T)}}$ the restriction of $p_*$ to $\beta(E/T)$ from Definition \ref{def:principal-over}.
\[\begin{tikzcd}
	E & {\beta(E/T)} & {\beta(E)} \\
	T & {\beta(T)}
	\arrow["p"', from=1-1, to=2-1]
	\arrow["\subseteq"{description}, draw=none, from=1-2, to=1-3]
	\arrow["{q}"', dashed, from=1-2, to=2-2]
	\arrow["{p_*}", from=1-3, to=2-2]
	\arrow["\delta"', from=2-1, to=2-2]
\end{tikzcd}\]
    Moreover, the fiber of $q$ at $\mu\in\beta(T)$ is given by the ultraproduct $\int_{x:\mu}E_x$, where $E_x$ denotes the fiber of $p$ at $x$.
\end{lem}
\begin{proof}
    We first show that $q$ is étale. Let $\mu$ be in $\beta(E/T)$, by hypothesis we can find a $\mu$-large subset $E_0$ on which $p$ is injective, then $\beta(E_0)\subseteq\beta(E)$ is an open contained in $\beta(E/T)$ and $p$ induces a homeomorphism $\beta(E_0)\to\beta(p(E_0))$.
    
    The fiber of $q$ over the principal ultrafilter $\delta_x$ is given by the ultrafilters on $E$ that are principal on an element of $E_x$, so we can identify the fiber of $q$ above $\delta_x$ with $E_x$. 

    Let $T_0\subseteq T$, we now look at the sections of $q$ on $\beta(T_0)$. The restriction of such a section to the principal ultrafilters gives a section $e : T_0\to E$ (by the description above) and as $\beta(E/T)$ is Hausdorff the section is determined by this $e$. Conversely, for an arbitrary section $e :T_0 \to E$, the map sending $\mu\in\beta(T_0)$ to $e_*(\mu)$ induces a section of $q$ on $\beta(T_0)$ ($p$ is injective on $e(T_0)$ that is $e_*(\mu)$-large). Hence, sections of $q$ on the basic open $\beta(T_0)$ correspond to sections of $p$ on $T_0 = \delta^{-1}(\beta(T_0))$, showing that $q$ is indeed the pushforward of $p$ along $\delta$.

    For the last claim, the fiber of $q$ above an arbitrary ultrafilter $\mu\in\beta(T)$ is given by the filtered colimit of the set of sections of $q$ on the $\beta(T_0)$ for $T_0$ a $\mu$-large subset, and sections of $q$ on $\beta(T_0)$ correspond to sections of $p$ on $T_0$ that are elements of $\prod_{x:T_0}E_x$.
\end{proof}

\begin{lem}\label{lem:étale-ucvg}
    Let $A$ be a Hausdorff topological space and $K$ be a compact subset of $A$. Then for any sheaf on $A$, a section on $K$ can be extended to a section on an open containing $K$.
\end{lem}
\begin{proof}
    This is a standard sheaf-theoretical fact; a proof can be found, for example, in \cite[Proposition 10.1]{extsheaf}.
\end{proof}

We now have all we need to prove Theorem \ref{Thm:étale-ucvg}.

\begin{proof}[Proof of Theorem \ref{Thm:étale-ucvg}]

The direct implication is straightforward. Suppose $p$ is étale and $\mu\cgeq p(e)$. We take a homeomorphism $U\to p(U)$ from an open neighborhood $U$ of $e$ and denote by $\xi$ its inverse. Any $\nu\cgeq e$ above $\mu$ is in the open $U$ and so $\nu = \xi_*(p_*(\nu)) = \xi_*(\mu)$ showing the uniqueness of the lifting. And by continuity $\xi_*(\mu)$ converges indeed to $\xi(p(e)) = e$, and $\xi_*(\mu)\in\beta(E/T)$ as $p$ is injective on $U$.

The converse is more demanding. Suppose $p : E \to T$ satisfies conditions \ref{cond1} and \ref{cond2} and fix some $e\in E$, we want to construct a homeomorphism from an open neighborhood of $e$ onto a neighborhood of $a \coloneqq p(e)$.
The hypothesis \ref{cond1} gives a set section $\sigma_a : \beta_a(T)\to\beta(E)$, where $\beta_a(T)$ is the subspace of $\beta(T)$ of ultrafilters converging to $a$, mapping $\mu\cgeq a$ to the unique $\sigma_a(\mu)$ over $\mu$ converging to $e$, and the hypothesis \ref{cond2} ensures that actually $\sigma_a$ factorizes through $\beta(E/T)$.
\[\begin{tikzcd}
    && {\beta(E/T)} \\
    {\beta_a(T)} && {\beta(T)} \\
    \arrow["p_*"', from=1-3, to=2-3]
    \arrow[color={blue}, "\sigma_a"{pos=0.2}, curve={height=-15pt}, dashed, from=2-1, to=1-3]
    \arrow[hook, from=2-1, to=2-3]
\end{tikzcd}\]
\begin{nclaim}
    This section is continuous.
\end{nclaim}
\begin{proof}[Proof of the claim]
    We use Lemma \ref{charact:map}.
    If $\mu\in\beta(\beta_a(T))$ then $\mu=(\nu)_{\nu:\mu}$ and $\mu$ converge to $\int_{\nu:\mu}\nu$ in $\beta_a(T)$. Then $\int_{\nu:\mu}\sigma_a(\nu)$ is above $\int_{\nu:\mu}\nu$ (as $p$ is continuous) and converges to $e$ (as each $\sigma_a(\nu)\cgeq e$), so by \ref{cond1} we get $\sigma_a(\int_{\nu:\mu}\nu) = \int_{\nu:\mu}\sigma_a(\nu)$.
\end{proof}

We would like to glue together all these lifts $(\sigma_a(\mu))$ to get a local section of $p$ around $a$; the crucial ingredient is given by Lemma \ref{lem:étale-ucvg}, we apply it to the sheaf $\beta(E/T)\stackrel{\text{ét}}\to\beta(T)$ of Lemma \ref{lem:principal-over} and the closed subset $\beta_a(T)$. 
We get a continuous section $\zeta$ on an open $V$ of $\beta(T)$, such that any $\mu$ converging to $a$ is in $V$ and satisfies $\zeta(\mu)\cgeq e$.

\[\begin{tikzcd}
    && {\beta(E/T)} \\
    {\beta_a(T)} && {\beta(T)} \\
    & V & {}\\
    \arrow["\et", from=1-3, to=2-3]
    \arrow[color={blue}, "\sigma_a"{pos=0.2}, curve={height=-15pt}, dashed, from=2-1, to=1-3]
    \arrow[hook, from=2-1, to=2-3]
    \arrow[color={blue}, hook, from=2-1, to=3-2]
    \arrow[color={blue}, "\zeta"{pos=0.2}, curve={height=-6pt}, dashed, from=3-2, to=1-3]
    \arrow[color={blue}, "{\mathrm{open}}"{description}, hook, from=3-2, to=2-3]
\end{tikzcd}\]

To get a section on $T$ we pull back (in $\Set$) along $\delta : T\to\beta(T)$. We get the following diagram
\[\begin{tikzcd}
    && {\beta(E/T)} \\
    {\beta_a(T)} && {\beta(T)} & E \\
    & V & {} & T \\
    && {\delta^*V}
    \arrow["\et", "p_*"', from=1-3, to=2-3]
    \arrow[color={blue}, "\sigma"{pos=0.2}, curve={height=-15pt}, dashed, from=2-1, to=1-3]
    \arrow[hook, from=2-1, to=2-3]
    \arrow[color={blue}, hook, from=2-1, to=3-2]
    \arrow["\delta"{description, pos=0.7}, dotted, from=2-4, to=1-3]
    \arrow["\lrcorner"{anchor=center, pos=0.125, rotate=-90}, draw=none, from=2-4, to=3-3]
    \arrow[from=2-4, to=3-4]
    \arrow[color={blue}, "\zeta"{pos=0.2}, curve={height=-6pt}, dashed, from=3-2, to=1-3]
    \arrow[color={blue}, "{\mathrm{open}}"{description}, hook, from=3-2, to=2-3]
    \arrow["\delta"{description, pos=0.7}, dotted, from=3-4, to=2-3]
    \arrow["\lrcorner"{anchor=center, pos=0.125, rotate=135}, draw=none, from=4-3, to=2-3]
    \arrow[color={blue}, "\xi"{pos=0.2}, curve={height=-6pt}, dashed, from=4-3, to=2-4]
    \arrow["\delta"{description, pos=0.7}, dotted, from=4-3, to=3-2]
    \arrow[color={blue}, hook, from=4-3, to=3-4]
\end{tikzcd}\]
where $\delta^*(V)$ is the subset of $x\in T$ such that $\delta_x\in V$, and $\xi : \delta^*V\to E$ is the unique map such that $\zeta(\delta_x) = \delta_{\xi(x)}$.

As $\zeta(\delta_a) = \delta_e$ we get $\xi(a) = e$, so $\xi$ seems to be a good candidate for our local homeomorphism. However, except in trivial cases, $\delta : T\to\beta(T)$ is not continuous, so there is no reason for $\delta^*(V)$ to be open, nor for $\xi$ to be continuous.
We thus restrict $\xi$ to $W$ the subset of points of $\delta^*V$ where $\xi$ is locally defined and sequentially continuous, \ie
\[W\coloneqq\{w\in \dom(\xi) \ |\ (\forall \mu\cgeq w)\ \mu\in\dom(\xi) \text{ and } \xi_*(\mu)\cgeq\xi(w) \}.\]

\begin{nclaim}
    $W$ is an open neighborhood of $a$ and $\xi$ is a continuous open map on $W$.
\end{nclaim}
\begin{proof}[Proof of the claim] \leavevmode
    \begin{itemize}
    
    \item \ul{$\xi$ is continuous and open:} $W$ has been chosen to make $\xi$ satisfy the hypothesis of Lemma \ref{charact:map} on it. We use Lemma \ref{charact:openmap} to see that $\xi$ is open: if $\nu\cgeq \xi(x)$ then by hypothesis $\nu = \xi_*(p_*(\nu))$.
    
    \item \ul{$W$ contains $a$:} let $\mu\cgeq a$, then $(\delta_x)_{x:\mu}\cgeq \mu$ and $\mu\in V$ open, so $(\delta_x)_{x:\mu}\in V$, \ie $\xi$ is $\mu$-defined, moreover $\xi_*(\mu) = \int_{x:\mu}\delta_{\xi(x)} = \int_{x:\mu}\zeta(\delta_x) = \zeta(\mu) \cgeq e = \xi(a)$.
    
    \item \ul{$W$ is open:} We take $\mu\cgeq w$ with $w\in W$ and we want to show that $\mu\in W$, \ie  
    \[(\forall x:\mu)(\forall \nu\cgeq x)\ \nu\in\dom(\xi) \land \xi_*(\nu)\cgeq\xi(w),\]
    and by Proposition \ref{prop:ultraquantification}\ref{prop:ultraquantification2} this is equivalent to
    \[(\forall (\nu_x\cgeq x)_{x:\mu})(\forall x:\mu)\ \nu_x\in\dom(\xi) \land \xi_*(\nu_x)\cgeq\xi(w).\]
    Let us fix some $(\nu_x\cgeq x)_{x:\mu}$.
    First, as $\int_{x:\mu}\nu_x\in\dom(\xi)$, we have $\nu_x\in\dom(\xi)$ for $\mu$-all $x$. We now show that $(\forall x:\mu)\ \xi_*(\nu_x)\cgeq\xi(x)$;
    for any $x$ in $\mu$, let us denote by $\sigma_x(\nu_x)\cgeq\xi(x)$ the unique lift of $\nu_x\cgeq x$ given by \ref{cond1}. On the one hand, as $\xi_*(\mu)\cgeq \xi(w)$, the ultrafilter $\int_{x:\mu}\sigma_x(\nu_x)$ converges to $\xi(w)$  and lies above $\int_{x:\mu}\nu_x$. On the other hand, as $\xi_*(\int_{x:\mu}\nu_x) \cgeq \xi(w)$, this is also the case of the ultrafilter $\int_{x:\mu}\xi_*(\nu_x)$. Hence, by \ref{cond1}, $\int_{x:\mu}\sigma_x(\nu_x) = \int_{x:\mu}\xi_*(\nu_x)$, and so, for $\mu$-all $x$, $\xi_*(\nu_x) = \sigma_x(\nu_x)\cgeq\xi(x)$. \qedhere
    \end{itemize}
\end{proof}

This shows together that $\xi(W)$ is an open and $\xi(W) \stackrel{p}\to W$ is a homeomorphism with inverse $\xi$. As $\xi(a) = e$ this exhibits the desired local homeomorphism and shows that $p$ is étale.
\end{proof}

\section{Ultracategories}\label{sec:ultracat}

In this section we fix the definition of ultracategories we will work with. This section is not necessary for our final result and is needed only for examples.

Makkai's original aim for introducing ultracategories was to prove conceptual completeness of coherent logic: whereas the category of models of a coherent theory is not enough to recover the theory, its ultracategory of models is.

We proposed in the introduction a topological understanding of it: the ultraproduct operation is a categorification of the ultraconvergence in compact Hausdorff spaces, and Makkai's result is a categorification of the fact that a (coherent) topological space can be recovered from its ultraconvergence relation. Following this intuition, Lurie noticed that ultracategories decategorify in $\beta$-algebras (or equivalently in compact Hausdorff spaces). This suggests a definition of ultracategories as $\bb$-pseudoalgebras and Hamad showed in \cite{Hamad} that Lurie's ultracategories are actually normal colax algebras for $\bb$ (while $\bb$-pseudoalgebras are Lurie's ultracategories satisfying one more condition, see Remark \ref{rmk:strong-weak}).

We also proposed another, more algebraic, justification: the ultrastructure of the category of models of a pretopos arises naturally by looking at the pretopos endofunctors on $\Set$. A slight extension of a result from Joyal \cite{Joyal} gives an equivalence $\Pretop(\Set^S,\Set^T)\simeq \Ind(\UMat(S,T))$, where $\UMat(S,T) = \bb(S)^T$ is the category of ultramatrices defined in \cite[Section 5.2]{GarnerUF}. Concretely, a functor from $\Set^S$ to $\Set^T$ that preserves the operations of first-order logic is canonically a filtered (large) colimit of ultraproducts. Hence, up to filtered colimits, the ultraproduct operations are ‘‘all we can do’’ when manipulating models of a coherent theory. 

On account of the above, it would be suitable to introduce ultracategories either by defining the pseudomonad structure on $\bb$ and computing its pseudoalgebras, or by looking at $\Pretop(\Set^S,\Set^T)$ and going through ultramatrices. However, we will let the algebraic study of ultracategories for future work, and simply deduce the definition of ultracategory by looking carefully at the behavior of the ultraproduct operations in $\Set$.

\begin{cons}
     Ultraproducts of sets behave like a tensor operation: we have associators and a unitor
    \begin{equation}\label{eq:ass-unit}
    \begin{tikzcd}[ampersand replacement=\&]
        \& {\Set^{\mu}} \\
        {\Set^{\sum_{s:\mu}\nu_s}} \& \& {\Set}
        \arrow["{\int_{\mu}}", from=1-2, to=2-3]
        \arrow["{(\int_{\nu_s})}", from=2-1, to=1-2]
        \arrow[""{name=0, anchor=center, inner sep=0}, "{\int_{\sum_{s:\mu}\nu_s}}"', from=2-1, to=2-3]
        \arrow["\wr", "{a_{(\nu_s)_{s:\mu}}}"', shorten <=7pt, shorten >=7pt, Rightarrow, from=0, to=1-2]
    \end{tikzcd} \hspace{30pt}
    \begin{tikzcd}[column sep=small,ampersand replacement=\&]
        \& \phantom{\Set^{\mu}} \\
        {\Set^{\ptuf}} \&\& {\Set}
        \arrow[""{name=0, anchor=center, inner sep=0}, "{\int_{\ptuf}}"'{pos=0.5}, from=2-1, to=2-3]
        \arrow["\sim", curve={height=-40pt}, from=2-1, to=2-3]
        \arrow["\wr", "u"', shift left=0.5, shorten <=7pt, shorten >=7pt, Rightarrow, from=0, to=1-2]
    \end{tikzcd}
    \end{equation}
    where we used, implicitly, the curryfication isomorphism from Construction \ref{cons:curry} to define the associator. 
    \[\Set^{\sum_{s:\mu}\nu_s}\simeq\int_{s:\mu}\Set^{\nu_s} \stackrel{(\int_{\nu_s})}\longrightarrow\int_{s:\mu}\Set=\Set^{\mu}\]
    Note that the associator corresponds to a currification: for $(A_{s,t})_{(s,t):\sum_{s:\mu}\nu_s}$ an ultrafamily of sets, the associator gives a bijection
    $\int_{\sum_{s:\mu}\nu_s}A_{s,t} \stackrel{\sim}\to\int_{s:\mu}\int_{t:\nu_s}A_{s,t}$.
    Moreover, these natural isomorphisms satisfy the following three usual coherences:
    \begin{equation}\label{eq:ass-coh}
    \begin{tikzcd}[ampersand replacement=\&]
        {\int_{(s,t,u):\mu\otimes\nu_s\otimes\lambda_{s,t}}A_{s,t,u}} \& {\int_{s:\mu}\int_{(t,u):\nu_s\otimes\lambda_{s,t}}A_{s,t,u}} \\
        {\int_{(s,t):\mu\otimes\nu_s}\int_{u:\lambda_{s,t}}A_{s,t,u}} \& {\int_{s:\mu}\int_{t:\nu_s}\int_{u:\lambda_{s,t}}A_{s,t,u}}
        \arrow["a", from=1-1, to=1-2]
        \arrow["a", from=1-1, to=2-1]
        \arrow["a", from=1-2, to=2-2]
        \arrow["a", from=2-1, to=2-2]
    \end{tikzcd}
    \end{equation}

    \begin{equation}\label{eq:unit-coh}
    \begin{tikzcd}[ampersand replacement=\&]
        {\int_{s:\mu}A_s} \& {\int_{\ptuf}\int_{s:\mu}A_s} \\
         \& {\int_{s:\mu}A_s}
        \arrow["u", from=1-1, to=1-2]
        \arrow[equals, from=1-1, to=2-2]
        \arrow["a", from=1-2, to=2-2]
    \end{tikzcd} \hspace{30pt}
    \begin{tikzcd}[ampersand replacement=\&]
        {\int_{s:\mu}A_s} \&  \\
        {\int_{s:\mu}\int_{\ptuf}A_s} \& {\int_{s:\mu}A_s.}
        \arrow["u", from=1-1, to=2-1]
        \arrow[equals, from=1-1, to=2-2]
        \arrow["a", from=2-1, to=2-2]
    \end{tikzcd}
    \end{equation}

    There are also \emph{reindexing operations} that give the ultraproduct an \emph{additive} flavor (in the sense that there are diagonal maps $A\to A^{\mu}$).
    For any map of ultrafilters $f : \lambda\to\kappa$, there is a natural transformation $f^{\#} : \int_{\kappa}\Rightarrow\int_{\lambda}\circ f^*$
    \begin{equation}\label{eq:reindexing}
    \begin{tikzcd}[ampersand replacement=\&]
    	{\Set^{\kappa}} \&\& {\Set^{\lambda}} \\
    	\& \Set
    	\arrow["{f^*}", from=1-1, to=1-3]
    	\arrow[""{name=0, anchor=center, inner sep=0}, "{\int_{\kappa}}"', from=1-1, to=2-2]
    	\arrow[""{name=1, anchor=center, inner sep=0}, "{\int_{\lambda}}", from=1-3, to=2-2]
    	\arrow["{f^{\#}}"', shift left=2, shorten <=8pt, shorten >=8pt, Rightarrow, from=0, to=1]
    \end{tikzcd}
    \end{equation}
    whose value on a $\kappa$-family $(A_k)_{k:\kappa}$ is given by
    \[f^{\#}_{(A_k)}(a_k)_{k:\kappa}\coloneqq(a_{f(l)})_{l:\lambda}.\]
    Note that by looking at $f : \mu\to\ptuf$ we recover the diagonal inclusion $(f^{\#})_A : A\hookrightarrow A^{\mu}$.
    Moreover, these $f^{\#}$ are strictly functorial in $f$, \ie 
    \begin{equation}\label{eq:reindexing-fctorial}
    \id_{\kappa}^{\#} = \id_{\int_{\kappa}} \hspace{10pt} \text{ and } \hspace{10pt} (f\circ g)^{\#} = g^{\#}\circ f^{\#}
    \end{equation}
    and the associators of \eqref{eq:ass-unit} are stable under the reindexing action, \ie
    \begin{equation}\label{eq:reindexing-ass}
    \begin{tikzcd}[ampersand replacement=\&]
    	{\int_{(k,t):\sum_{k:\kappa}\nu_k}A_{k,t}} \& {\int_{k:\kappa}\int_{t:\nu_k}A_{k,t}} \\
    	{\int_{(l,t):\sum_{l:\lambda}\nu_{f(l)}}A_{f(l),t}} \& {\int_{l:\lambda}\int_{t:\nu_{f(l)}}A_{f(l),t}}
    	\arrow["a", from=1-1, to=1-2]
    	\arrow["{f^{\#}}"', from=1-1, to=2-1]
    	\arrow["{(f,\id)^{\#}}", from=1-2, to=2-2]
    	\arrow["a"', from=2-1, to=2-2]
    \end{tikzcd}
    \end{equation}
    commutes; this can also be written as the commutation of the following diagram.
\[\begin{tikzcd}[ampersand replacement=\&]
	{\Set^{\sum_{k:\kappa}\nu_{k}}} \& {} \&\& {} \& {\Set^{\sum_{l:\lambda}\nu_{f(l)}}} \\
	\&\& \Set \\
	{\Set^{\kappa}} \& {} \&\& {} \& {\Set^{\lambda}}
	\arrow["{(f,\id)^*}", dashed, from=1-1, to=1-5]
	\arrow[""{name=0, anchor=center, inner sep=0}, "\int"{description}, from=1-1, to=2-3]
	\arrow["{(\int_{\nu_{k}})}"', from=1-1, to=3-1]
	\arrow["a"', shift right=7, shorten <=19pt, shorten >=23pt, Rightarrow, from=1-2, to=3-2]
	\arrow["a", shift left=7, shorten <=19pt, shorten >=23pt, Rightarrow, from=1-4, to=3-4]
	\arrow[""{name=1, anchor=center, inner sep=0}, "\int"{description}, from=1-5, to=2-3]
	\arrow["{(\int_{\nu_{f(l)}})}", from=1-5, to=3-5]
	\arrow[""{name=2, anchor=center, inner sep=0}, "{\int_{\kappa}}"{description}, from=3-1, to=2-3]
	\arrow["{f^*}"', dashed, from=3-1, to=3-5]
	\arrow[""{name=3, anchor=center, inner sep=0}, "{\int_{\lambda}}"{description}, from=3-5, to=2-3]
	\arrow["{(f,\id)^{\#}}"{description}, shift left, shorten <=14pt, shorten >=14pt, Rightarrow, dashed, from=0, to=1]
	\arrow["{f^{\#}}"{description}, shift right=2, shorten <=14pt, shorten >=14pt, Rightarrow, dashed, from=2, to=3]
\end{tikzcd}\]
\end{cons}

The notion of ultracategory abstracts the above properties of the functors $\int_{\mu} : \Set^{\mu}\to\Set$.
\begin{defn}
    An \emph{ultracategory} is given by a category $\M$ together with
    \begin{enumerate}
        \item functors $\int_{\mu} : \M^{\mu}\to\M$ for each ultrafilter $\mu$; 
        \item associators and a unitor, given by natural isomorphisms as in \eqref{eq:ass-unit}, satisfying the same three coherences \eqref{eq:ass-coh}, \eqref{eq:unit-coh};
        \item and reindexing transformations $f^{\#} : \int_{\kappa}\Rightarrow\int_{\lambda}\circ f^*$, as in \eqref{eq:reindexing}, for any map of ultrafilters $f\in\UF(\lambda,\kappa)$;
        \item that are strictly functorial in $f$ and respect the associators, \ie the analogue of \eqref{eq:reindexing-fctorial} holds and the analogue of \eqref{eq:reindexing-ass} commutes.
    \end{enumerate}
\end{defn}

\begin{rmk}
    The strict functoriality of the reindexing transformations shows that the functors $\int_{\mu} : \M^{\mu}\to\M$ and $\int_{\nu} : \M^{\nu}\to\M$ are equivalent for isomorphic ultrafilters $\mu$ and $\nu$. 
\end{rmk}

\begin{defn}
    A \emph{left-ultrafunctor} from an ultracategory $\M$ to another $\N$ is given by a functor of the underlying categories $F : \M\to\N$, together with natural transformations
\[\begin{tikzcd}
	{\M^{\mu}} & {\N^{\mu}} \\
	\M & \N
	\arrow["{F^{\mu}}", from=1-1, to=1-2]
	\arrow[""{name=0, anchor=center, inner sep=0}, "{\int_{\mu}}"', from=1-1, to=2-1]
	\arrow[""{name=1, anchor=center, inner sep=0}, "{\int_{\mu}}", from=1-2, to=2-2]
	\arrow["F"', from=2-1, to=2-2]
	\arrow["\sigma_{\mu}"', shorten <=15pt, shorten >=15pt, Rightarrow, from=0, to=1]
\end{tikzcd}\]
    for all ultrafilters $\mu$, that commute with the associators, the unitors, and the reindexings.
    Left-ultrafunctors from $\M$ to $\N$ form a category, where maps are natural transformations between the underlying functors of categories that commute with the $(\sigma_{\mu})$. Ultracategories with left-ultrafunctors form a strict 2-category that we denote by $\UltL$. An \emph{ultrafunctor} is a left-ultrafunctor with all the $\sigma_{\mu}$ being isomorphisms, ultracategories with ultrafunctors form a sub-2-category of $\UltL$ that we denote by $\Ult$.
\end{defn}

\begin{rmk}\label{rmk:strong-weak}
    We expect this definition to be equivalent to the one of Rosolini--Garner \cite{Pino} (\ie to be exactly the $\bb$-pseudoalgebras); Lurie's definition is slightly weaker as he does not ask for the associator to be an isomorphism. We suggest calling Lurie's notion \emph{weak ultracategory} and the one above \emph{ultracategory}, or \emph{strong ultracategory} if we want to emphasize that the associator is indeed an isomorphism. This is closer to Makkai's original motivation to have ultraproduct operations that behave as much as possible like the ultraproducts of sets. The main example of weak ultracategories that are not ultracategories is what Lurie calls the ``categorical ultrastructure'' \cite[Example 1.3.8]{Lurie} on a category with products and filtered colimits, and we indeed believe that in general these categorical (weak) ultrastructures should not be thought of as a categorified compact Hausdorff space.
\end{rmk}

\begin{ex}
    The category $\Set$ with the usual ultraproduct of sets has a structure of ultracategory. More generally, the category of models of any small pretopos has a structure of an ultracategory.
\end{ex}

\begin{ex}\label{ex:bb-ultracat}
    For $S$ a (large) set, the category $\bb(S) = \UF^{\op}_S$ has an ultracategory structure: the ultraproduct of a $\mu$-family of objects $((x_{s,t})_{t:\nu_s})_{s:\mu}$ is given by $(x_{s,t})_{(s,t):\sum_{s:\mu}\nu_s}$.
    Note that $\bb(S)$ is isomorphic as a category to the disjoint union of the slice categories $\UF^{\op}/\mu$ with $\mu$ varying over small ultrafilters on $S$, but the ultracategory structure makes the different connected components interact.
\end{ex}

The main result of Lurie's paper \cite[Theorem 2.2.2]{Lurie} can be stated as follows.
    
\begin{Thm}[\cite{Lurie}]\label{Thm:Lurie}
    The points of a coherent topos $\E$ have an ultracategory structure $\vpt(\E)$, and the evaluation functor $\E\to\CAT(\pt(\E),\Set)$ can be enhanced into an equivalence
    \[\E\stackrel{\sim}\longrightarrow\UltL(\vpt(\E),\Set).\]
\end{Thm}

As explained in \cite[Section 2.3]{Lurie}, Lurie's theorem lets us recover Makkai's original results. We will, in turn, recover Lurie's theorem as a by-product of our final theorem (Remark \ref{rmk:get-Lurie-back}).

\section{Virtual ultracategories}\label{sec:vultracat}

The definition of virtual ultracategory is a straightforward categorification of the definition of relational $\beta$-module (Definition \ref{def:rel-beta-mod}).

\begin{defn}
    A \emph{virtual ultracategory} (abbreviated v-ultracategory) $X$ is given by:
    \begin{enumerate}
        \item a class of objects $X_0$;
        \item a homset functor \[\Hom_X : X_0\times\bb(X_0)\longrightarrow \Set.\]
        We will write $X(a,(b_s)_{s:\mu})$ for the set $\Hom_X(a,(b_s)_{s:\mu})$
        and $f:a \ultrato (b_s)_{s:\mu}$ for an element $f\in X(a,(b_s)_{s:\mu})$, such an $f$ will be called an \emph{ultraarrow of type $\mu$} (or a \emph{$\mu$-arrow}) with domain $a$ and codomain $(b_s)_{s:\mu}$;
        \item an identity $\ptuf$-arrow \[\id_a : a\ultrato (a)_{\ptuf}\] for any object $a$ in $X_0$;
        \item a composition operation
        \[\int_{s:\mu}\Hom(b_s,(c_{s,t})_{t:\nu_s})\times \Hom(a,(b_s)_{s:\mu}) \longrightarrow \Hom(a,(c_{s,t})_{(s,t):\sum_{s:\mu}\nu_s})\] natural in $(b_s)_{s:\mu}\in\bb(X_0)$ and $((c_{s,t})_{t:\nu_s})_{s:\mu}\in\bb(X_0)^{\mu}$.
        We will denote by
        \[(g_s)_{s:\mu}\circ f : a \ultrato (c_{s,t})_{(s,t):\sum_{s:\mu}\nu_s}\]
        the composite of $f : a \ultrato (b_s)_{s:\mu}$ and $(g_s)_{s:\mu} \in \int_{s:\mu} X(b_s,(c_{s,t})_{t:{\nu_s}})$;
        \item satisfying the usual identity and associativity laws
        \[(g)_{s:\ptuf}\circ \id_a = g \text{ , } (\id_{b_s})_{s:\mu}\circ f = f,\]
        \[ (h_{s,t})_{(s,t):\sum_{s:\mu}\nu_s} \circ ((g_s)_{s:\mu}\circ f) = ((h_{s,t})_{t:\nu_s} \circ (g_s))_{s:\mu}\circ f,\]
        where we implicitly use Construction \ref{cons:ass-unit-sum}.
    \end{enumerate}
\end{defn}

\begin{ex}\label{ex:vUlt} \leavevmode
    \begin{enumerate}[label=(\roman*)]
        \item\label{ex:vUlt-pt} There is a trivial v-ultracategory with only one object and whose all homsets are singleton, we will denote it $\mathbb{1}$ and call it \emph{the point v-ultracategory}.
        
        \item The v-ultracategory of small sets is a particularly important one, we denote it again by $\Set$. Its objects are small sets and its homsets are given by
        \[\Set(A,(B_s)_{s:\mu}) \coloneqq \Set\left(A,\int_{s:\mu}B_s\right),\]
        the unit and the composition are defined using the isomorphisms of \eqref{eq:ass-unit}.
        
        \item\label{ex:vUlt-topoi} In the same way that the points of a topological space form a relational $\beta$-module, there is a v-ultrastructure on the points of any topos. Recall that in a topological space $a\cleq\mu$ means that any open containing $a$ contains also $\mu$, formally written as
        \[\forall\text{ $U$ open of $T$ }  (a\in U \implies \forall x:\mu, x\in U).\]
        For a topos $\E$, we construct a v-ultracategory $\vpt(\E)$ on the points of $\E$ by taking a proof relevant version of the above formula
        \[\vpt(\E)(a,(b_s)_{s:\mu}) \coloneqq \Nat_{E:\E}\left(E_a,\int_{s:\mu}E_{b_s}\right),\]
        where $U$ open of $T$ is replaced by $E$ sheaf of $\E$, the proposition $(a\in U)$ is replaced by the fiber $E_a$, and the ultraquantification is replaced by an ultraproduct of sets.
        
        Concretely, ultraarrows from $a$ to $(b_s)_{s:\mu}$ are given by natural transformations as below.
        \[\begin{tikzcd}[sep=small]
        	& {} & {\Set^{S}} \\
        	\E \\
        	& {} & \Set
        	\arrow["{\int_{\mu}}", from=1-3, to=3-3]
        	\arrow["{({b_s}^*)_{s:S}}"{pos=0.7}, from=2-1, to=1-3]
        	\arrow["{a^*}"'{pos=0.6}, from=2-1, to=3-3]
        	\arrow[shift right=4, shorten <=17pt, shorten >=12pt, Rightarrow, from=3-2, to=1-2]
        \end{tikzcd}\]
        In general $\int_{s:\mu}b_s^*$ is not cocontinuous, so it does not induce a point of $\E$, however $a^*$ is cocontinuous, so such a natural transformation is determined by its value on some small generating set of the topos $\E$, and so $\Nat_{E:\E}(E_a,\int_{s:\mu}E_{b_s})$ is small.
        
        Note that we recover $\mathbb{1}$ (resp. $\Set$) from above as the v-ultracategory of points of the terminal topos (resp. the classifying topos of the theory of objects).
        
        \item Generalizing the above example, for $\iota : X\to \pt(\E)$ a (large) family of points of a topos $\E$, we define a v-ultracategory $\vpt(\E;X)$ with class of objects $X$ by pulling back $\vpt(\E)$ along $\iota$. The homsets are given by the formula
        \[\vpt(\E;X)(a,(b_s)_{s:\mu}) \coloneqq \Nat_{E:\E}\left(E_{\iota(a)},\int_{s:\mu}E_{\iota(b_s)}\right).\]
        
        \item\label{ex:vUlt-topsp} One can associate to a topological space $T$ a v-ultracategory $\vpt(T)$ on the set of points of $T$. The homsets are given by
        \[\vpt(T)(a,(b_s)_{s:\mu}) \coloneqq \llbracket a\cleq (b_s)_{s:\mu}  \rrbracket.\]
        This construction corresponds to $\vpt(\Sh(T);T)$.
        
        \item\label{ex:vUlt-cat} From a locally small category $\C$, we can define a v-ultracategory $\Alex(\C)$ on the objects of $\C$, where the homsets are defined by
        \[\Alex(\C)(a,(b_s)_{s:\mu}) \coloneqq \int_{s:\mu}\C(a,b_s),\]
        if $\C$ is small this construction coincides with $\vpt(\Set^{\C};\C_0)$ by Yoneda.
        Note that this construction generalizes the Alexandroff topology: if the category is posetal the v-ultracategory obtained is the same as the one obtained by applying the previous construction to its Alexandroff topology.
    
        \item\label{ex:vUlt-ultracat} From any ultracategory $\M$, we can define a v-ultracategory $\underline{\M}$ having the same objects as $\M$, and
        \[\underline{\M}(a,(b_s)_{s:\mu}) \coloneqq \M\left(a,\int_{s:\mu}b_s\right),\]
        if $\M$ is the ultracategory of models of a coherent topos $\E$, then this v-ultracategory coincides with $\vpt(\E)$. 
        Note that we need $\M$ to be a strong ultracategory (Remark \ref{rmk:strong-weak}) to be able to define the composition in $\underline{\M}$.
    \end{enumerate}
\end{ex}

Virtual ultracategories are generalized multicategorical structures and, as one can expect, they form a strict 2-category.
\begin{defn}\leavevmode
    \begin{enumerate}
        \item A \emph{functor} between two v-ultracategories $F:X\to Y$ is given by a map at the level of objects $F_0 : X_0 \to Y_0$ and maps on the homsets $F_{a,(b_s)_{s:\mu}} : X(a,(b_s)_{s:\mu}) \to Y(F_0(a),(F_0(b_s))_{s:\mu})$ natural in $(b_s)\in\bb(X_0)$, such that
        \[F(\id_a) = \id_{F(a)} \text{ , and }\]
        \[F((g_s)_{s:\mu}\circ f) = (F(g_s))_{s:\mu} \circ F(f).\]
        \item A \emph{natural transformation} $\alpha : F\Rightarrow G$ between two such functors consists of $\ptuf$-arrows $\alpha_a : F(a)\ultrato (G(a))_{\ptuf}$ for each object $a\in X_0$, such that the following naturality condition is satisfied:
        \[G(f)\circ\alpha_a = (\alpha_{b_s})_{s:\mu}\circ F(f) \text{ , for any }f\in X(a,(b_s)_{s:\mu}).\]
    \end{enumerate} 
    This endows the collection of virtual ultracategories with a strict 2-categorical structure that we denote by $\vUlt$. 
    Note that in general $\vUlt(X,Y)$ is neither small nor locally small. 
\end{defn}

\begin{prop}\label{prop:vUlt-subcat}
    The $\vpt$ construction from Example \ref{ex:vUlt}\ref{ex:vUlt-topoi} induces a strict 2-functor \[\GTop \to \vUlt.\]
    The constructions of Examples \ref{ex:vUlt}\ref{ex:vUlt-topsp}--\ref{ex:vUlt-ultracat} induce three 2-fully faithful strict 2-functors
        \[\TopSp\hookrightarrow\vUlt , \CAT\hookrightarrow\vUlt \text{, and }\UltL\hookrightarrow\vUlt.\]
\end{prop}
\begin{proof}
    We give some details for the proof that $\UltL\hookrightarrow\vUlt$ is 2-fully faithful. Let $\M$ and $\N$ be two ultracategories and $F : \underline\M \to \underline\N$ a functor between v-ultracategories. The ultraarrow $\int_{s:\mu}a_s\ultrato(a_s)_{s:\mu}$ of $\underline\M$ corresponding to the identity is sent to an ultraarrow $F(\int_{s:\mu}a_s)\ultrato(F(a_s))_{s:\mu}$ of $\underline\N$ that corresponds to some arrow $\sigma : F(\int_{s:\mu}a_s)\to\int_{s:\mu}F(a_s)$ in $\N$. One can check that these $(\sigma)$ induce a left-ultrafunctor structure $F : \M \to \N$.
    
    Conversely, from a left-ultrafunctor $F : \M\to\N$ we can define a functor of v-ultracategories as follows: an ultraarrow $a \ultrato (b_s)_{s:\mu}$ corresponding to some $f: a \to \int_{s:\mu}b_s$ is sent to the ultraarrow $F(a) \ultrato (F(b_s))_{s:\mu}$ corresponding to the composite of $F(f):F(a)\to F(\int_{s:\mu}b_s)$ with the comparison map $F(\int_{s:\mu}b_s)\to\int_{s:\mu}F(b_s)$.
    
    One can then check that these two constructions establish a bijection between $\UltL(\M,\N)$ and $\vUlt(\underline{\M},\underline{\N})$.
\end{proof}

\begin{rmk}
    The final result (Corollary \ref{cor:final-adj}) will ensure that $\GTop \to \vUlt$ is actually 2-fully faithful when restricted to topoi with enough points.
\end{rmk}

\begin{defn}
    A \emph{point}  $x$ of a v-ultracategory $X$ is a functor $x : \mathbb{1} \to X$, where $\mathbb{1}$ is the point v-ultracategory from Example \ref{ex:vUlt}\ref{ex:vUlt-pt}. The category $\vUlt(\mathbb{1},X)$ is denoted $\pt(X)$ and called \emph{the category of points} of $X$.
    Explicitly, the category of points $\pt(X)$ has the same objects as $X$ and $\pt(X)(a,b)$ is given by $X(a,(b)_{\ptuf})$.
\end{defn}

\begin{ex}
    Following Examples \ref{ex:vUlt}\ref{ex:vUlt-topoi} and \ref{ex:vUlt-topsp}--\ref{ex:vUlt-ultracat}:
    \begin{enumerate}
        \item for $\E$ a topos, $\pt(\vpt(\E))$ is the usual category of points of $\E$;
        \item for $T$ a topological space, $\pt(\vpt(T))$ is the specialization order on $T$;
        \item for $\C$ a locally small category, $\pt(\Alex(\C))$ recovers the category $\C$;
        \item for $\M$ an ultracategory, $\pt(\underline{\M})$ is the underlying category of $\M$.
    \end{enumerate}
\end{ex}

\begin{rmk}
    As explained in Remark \ref{rmk:largebb}, the functor $\bb$ can be extended to $\CAT$; the homset functor $\Hom_X : X\times\bb(X) \to \Set$ can thus be extended to a functor $\pt(X)^{\op}\times \bb(\pt(X)) \to \Set$, \ie a distributor between $\pt(X)$ and $\bb(\pt(X))$, categorifying the ultraconvergence relation between points and ultrafilters in a topological space.
\end{rmk}

\begin{rmk}
    If $\M$ is an ultracategory, the functors $\vpt(\underline{\M})(-,(b_s)_{s:\mu}) : \pt(\M)^{\op} \to \Set$ are represented by $\int_{s:\mu}b_s$. We actually expect the notion of virtual ultracategory to fit into the framework of generalized multicategories from \cite{GenMultiCat} and that the essential image of $\UltL\hookrightarrow\vUlt$ is exactly the representable virtual ultracategories in the sense of \cite[Section 9]{GenMultiCat}.
\end{rmk}

\begin{rmk}
    The choice of generalizing the codomain and not the domain can be surprising at first, let us temporarily call ‘‘right-virtual ultracategory’’ the notion of v-ultracategory with ultraarrows having a generalized domain instead of a generalized codomain \ie with a hom-functor $\bb(X)^{\op}\times X\to\Set$. 
    Note that the argument for $\vpt(\E)$ of Example \ref{ex:vUlt}\ref{ex:vUlt-topoi} to have small homsets does not work for right-v-ultracategories.  
    
    Also, whereas our notion of v-ultracategories extends ultracategories with left-ultrafunctors, right-v-ultracategories extend ultracategories with right-ultrafunctors; moreover, we would then be able to embed all Lurie's ultracategories (see Remark \ref{rmk:strong-weak}) into right-v-ultracategories, and in particular all the categorical ultracategories \cite[Example 1.3.8]{Lurie}.
    The notion of right-ultrafunctors seems to have some importance in regard to Lurie's \emph{ultracategory envelopes} \cite[Section 8]{Lurie}, and we believe that the notion of right-ultrafunctors (and of right-v-ultracategories) makes more sense if one studies the algebraic aspect of the ultraproduct and sees the ultraproduct as an ‘‘ultratensor operation’’ rather than a ‘‘categorified ultraconvergence’’.
\end{rmk}

\section{Back and forth with topoi}\label{sec:back-forth}

We have seen in Example \ref{ex:vUlt}\ref{ex:vUlt-topoi} that the points of a topos form naturally a v-ultracategory. As explained in the introduction, our goal is to reconstruct the topos from its v-ultracategory of points. We now see that any object of a topos can be seen as a \emph{sheaf} over the v-ultracategory of points of the topos.

\begin{defn}
    An \emph{ultrasheaf} (or simply a \emph{sheaf}) over a v-ultracategory $X$ is a functor of v-ultracategories $A : X \to \Set$. The category $\vUlt(X,\Set)$ is \emph{the category of (ultra)sheaves over $X$} and is denoted by $\vSh(X)$. 
\end{defn}

Concretely, a sheaf $A$ over $X$ is given by a family of sets $(A_x)_{x:X_0}$ indexed by the objects of $X$, together with maps $A(f) : A_a\to\int_{s:\mu} A_{b_s}$ for each $f\in X(a,(b_s)_{s:\mu})$ satisfying the usual functoriality axioms.

\begin{defn}\label{def:ev-functor}
    Let $\E$ be a topos, we define the \emph{evaluation functor} $\ev : \E \to \vSh(\vpt(\E))$ by
    \[\ev : \E\simeq\GTop(\E,\SetO) \stackrel{\vpt}\longrightarrow\vUlt(\vpt(\E),\Set) = \vSh(\vpt(\E)).\]
    More generally, for $\iota : X\to\pt(\E)$ a (large) family of points, there is an evaluation functor $\ev : \E \to \vSh(\vpt(\E;X))$ given by the following composite:
    \[\E\simeq\GTop(\E,\SetO) \stackrel{\vpt}\longrightarrow\vUlt(\vpt(\E),\Set) \stackrel{}\longrightarrow\vUlt(\vpt(\E;X),\Set) = \vSh(\vpt(\E;X)).\]
    Concretely, for any $E\in\E$ and $x\in X$, \[\ev(E)(x) = \iota(x)^*(E)\]
    \ie $\ev(E)(x)$ is the stalk of $E$ at the point $\iota(x)$.
\end{defn}

\begin{ex}\label{ex:ultrasheaf} \leavevmode
    \begin{enumerate}[label=(\roman*)]
        \item For $\C$ a locally small category, the category $\vSh(\Alex(\C))$ is equivalent to the functor category $\Set^{\C}$.
        This follows from noticing that any map $(f_s)_{s:\mu}\in\Alex(\C)(a,(b_s)_{s:\mu})=\int_{s:\mu}\C(a,b_s)$ can be factorized as $a\ultrato(a)_{s:\mu}$ followed by $(f_s : a\to b_s)_{s:\mu}$.
        In the particular case when the category is a small category, we deduce that for $\E=\Set^{\C}$ a presheaf topos, the functor $\ev : \E \to \vSh(\vpt(\E;\C_0))$ is always an equivalence.
        
        \item\label{ex:ultrasheaf-ultracat} For $\M$ an ultracategory, the category $\vSh(\underline{\M})$ is equivalent to $\UltL(\M,\Set)$ the category of left-ultrafunctors from $\M$ into $\Set$. This follows from Proposition \ref{prop:vUlt-subcat} noticing that $\underline{\Set}$ is the v-ultracategory $\Set$.
        In particular, for $\E$ a coherent topos, the functor $\ev : \E \to \vSh(\vpt(\E))$ is the inverse image of the geometric morphism $\UltL(\pt(E),\Set)\to\E$ from \cite[Construction 2.2.1]{Lurie}, and thus Lurie's Theorem \ref{Thm:Lurie} can be rephrased as $\ev : \E \to \vSh(\vpt(\E))$ is an equivalence.
        \item For $T$ a topological space, the following proposition (Proposition \ref{prop:0dimcase}) shows that $\ev : \Sh(T) \to \vSh(\vpt(T))$ is an equivalence.
        Or equivalently, $\ev : \E \to \vSh(\vpt(\E;T))$ is an equivalence for $\E \coloneqq\Sh(T)$.
    \end{enumerate}
\end{ex}

\begin{prop}\label{prop:0dimcase}
    Let $T$ be a topological space, the functor $\ev : \Sh(T)\to\vSh(\vpt(T))$ is an equivalence.
\end{prop}
\begin{proof}
    The proof relies on the characterization of étale maps of Theorem \ref{Thm:étale-ucvg}.
    From an ultrasheaf $A\in\vSh(\vpt(T))$ we get a sheaf over the space $T$ by taking $E \coloneqq \bigsqcup_{a:T}A_a$ with the topology given by: $B\subseteq E$ is open, if
    \[\text{for all } a\cleq \mu \text{ , } A(\llbracket a\cleq\mu\rrbracket) : A_a \longrightarrow \int_{x:\mu}A_x \text{ restricts to } B_a \longrightarrow \int_{x:\mu}B_x,\]
    where $B_x \coloneqq B\cap A_x$.
    The projection map $E\to T$ satisfies the hypotheses of Theorem \ref{Thm:étale-ucvg} and so is \'etale.   
    
    One can then check that this construction induces a functor $\vSh(\vpt(T))\to \Sh(T)$ and that this functor is a pseudoinverse to $\ev : \Sh(T)\to\vSh(\vpt(T))$.
\end{proof}

These examples are particular instances of our main result: any topos $\E$ with enough points can be recovered as the category of ultrasheaves over the virtual ultracategory of its points. More precisely, we will show that for $X$ a separating class of points of $\E$, the functor $\ev : \E \to \vSh(\vpt(\E;X))$ is an equivalence.  Before going into the proof of this result, we study the exactness properties of the category of sheaves over a v-ultracategory.

\begin{lem}\label{lem:ultrasheaf-exactness}
    The category $\vSh(X)$ admits small colimits and finite limits, and they are computed objectwise.
\end{lem}
\begin{proof}
    Let $\eI$ be a small category and $(A^{i})_{i:\eI}$ an $\eI$-diagram in $\vSh(X)$.
    
    For any $(x_s)_{s:\mu}$ ultrafamily of objects of $X$, there is a comparison map 
    \[\colim_{i:\eI}\left(\int_{s:\mu}(A^{i})_{x_s}\right) \longrightarrow \int_{s:\mu}\colim_{i:\eI}((A^i)_{x_s})\]
    allowing us to define an ultrasheaf structure on $(\colim((A^i)_x))_{x:\Ob(X)}$ that we will denote by $A\in\vSh(X)$. There is a natural cocone of $A$ above the $(A^{i})$ and this cocone is universal as maps of ultrasheaves are given objectwise.

    For limits, the comparison map is in the wrong direction
    \[\lim_{i:\eI}\left(\int_{s:\mu}(A^{i})_{x_s}\right) \longleftarrow \int_{s:\mu}\lim_{i:\eI}((A^i)_{x_s})\]
    but this comparison map is bijective when $\eI$ is finite, allowing us to conclude as above.
\end{proof}

\begin{cor}
    The functor $\ev : \E \to \vSh(\vpt(\E;X))$ is lex cocontinuous.
\end{cor}

\begin{prop}
    The category $\vSh(X)$ is an infinitary pretopos.
\end{prop}
\begin{proof}
    The proof is essentially the same as \cite[Proposition 5.4.6]{Lurie}.
    Let $\alpha : A \to B$ be a map of ultrasheaves over $X$ and $(B^{i}\to B)_{i:\eI}$ an $\eI$-diagram in $\vSh(X)/B$ for $\eI$ a small category. We need to show that the canonical map
    \[\colim_{i:\eI}(B^i\times_{B}A)\longrightarrow \colim_{i:\eI}(B^i)\times_{B}A\]
    is an isomorphism of sheaves as it is a bijection on each object of $X$. This follows from the fact that small colimits and pullbacks are computed objectwise by Lemma \ref{lem:ultrasheaf-exactness} and that small colimits and pullbacks commute in $\Set$.

    Let now $R\subseteq A\times A$ be an equivalence relation in $\vSh(X)$.
    We need to show that the canonical map
    \[R\longrightarrow A \times_{A/R} A\]
    is an isomorphism, where $A/R \coloneqq \coeq(R\rightrightarrows A)$ is the quotient of $A$ by $R$.
    Again, this follows from the fact that the quotient $A/R$ and the pullback $A \times_{A/R} A$ are computed objectwise by Lemma \ref{lem:ultrasheaf-exactness} and that equivalence relations are effective in $\Set$.
\end{proof}

Hence the category $\vSh(X)$ only lacks accessibility to be a topos. In general $\vSh(X)$ is neither locally small nor accessible as shown by the example $\vSh(X)\simeq\Set^{X}$ with $X$ a discrete virtual ultracategory on a large set of objects. This motivates the following definition.

\begin{defn}
    A virtual ultracategory is said \emph{bounded} if its category of sheaves is accessible (and in particular locally small). We denote by $\vUltb$ the 2-full subcategory of bounded virtual ultracategories.
\end{defn}

\begin{rmk}\label{rmk:acc}
    Our notion of bounded v-ultracategory lacks a justification and might well be the wrong notion (we chose this definition as it is enough to express the 2-adjunction we aim for). We leave open the definition of a sensible notion of bounded (and accessible) v-ultracategories.
\end{rmk}

\begin{cor}
    The representable 2-functor \[\vSh = \vUlt(-,\Set) : \vUlt\longrightarrow\CAT^{\op}\] restricts to a 2-functor \[\vSh : \vUltb\longrightarrow\GTop.\]
\end{cor}
\begin{proof}
    Let $F : X\to Y$ be a functor between bounded v-ultracategories. Precomposition by $F$ gives a functor of categories $\vSh(Y)\to\vSh(X)$ and this functor is lex and cocontinuous by Lemma \ref{lem:ultrasheaf-exactness}. Hence, as $X$ and $Y$ are bounded, it induces a geometric morphism $\vSh(F) : \vSh(X)\to\vSh(Y)$. The rest of the 2-functorial structure is straightforward.
\end{proof}

\section{Representation by groupoids}\label{sec:desc}

The next two sections are dedicated to the proof of our reconstruction theorem in the case $X$ is a small separating set of points: we show the following proposition.

\begin{prop}\label{prop:small-sep-case}
    Let $X$ be a small separating set of points of a topos $\E$, the evaluation functor $\ev : \E \to \vSh(\vpt(\E;X))$ from Definition \ref{def:ev-functor} is an equivalence of categories.
\end{prop}
One can admit this proposition and skip to Section \ref{sec:ccl}.

The proof of Proposition \ref{prop:small-sep-case} will rely on representations of topoi by topological groupoids. Representing topoi by topological groupoids is a powerful and well-established method to reconstruct a topos from its points, by considering equivariant sheaves over some topological groupoid of its points. It was first introduced by Butz--Moerdijk in \cite{BM1} but we will use a more general version given by Wrigley in \cite{Josh}. 
In this section, we recall results about descent and representation by topological groupoids, we then explain our strategy to prove Proposition \ref{prop:small-sep-case} and give an essential result on descent of v-ultracategories in Proposition \ref{prop:desc-vUlt}.

In the following, we denote by $\K$ a strict 2-category with pseudopullbacks; in our case of interest, $\K$ will be $\TopSp$, $\GTop$, or $\vUlt$. 

\begin{cons}\label{cons:strict2pb}
    The 2-categories $\TopSp$ and $\vUlt$ have strict 2-pullbacks that can be constructed in a similar fashion to strict 2-pullbacks of categories.
    \begin{enumerate}
        \item For two continuous maps $f : \eX\to \eZ$ and $g : \eY\to \eZ$ of topological spaces, the 2-pullback of $f$ and $g$ is the subspace of $\eX\times \eY$ given by the pairs $(x,y)$ with $f(x)$ and $g(y)$ equivalent for the specialization order (\ie $f(x)$ and $g(y)$ are contained in the same opens).
        Note that this is slightly different from the usual 1-pullback of topological spaces, but that this difference disappears if $\eZ$ is T0-separated.
        \item For two functors of v-ultracategories $F : \eX\to \eZ$ and $G : \eY\to \eZ$, the 2-pullback is given by the v-ultracategory whose objects are triplets $(\theta,x,y)$ of an object $x\in \eX_0$, an object $y\in \eY_0$, and an isomorphism $\theta\in \eZ(F(x),G(y))$; an ultraarrow from $(\theta,x,y)$ to $(\theta_s,x_s,y_s)_{s:\mu}$ is given by a pair $(f,g)$ of an ultraarrow $f\in \eX(x,(x_s)_{s:\mu})$ and an ultraarrow $g\in \eY(y,(y_s)_{s:\mu})$ such that $G(g)\circ\theta = (\theta_s)_{s:\mu}\circ F(f)$.
    \end{enumerate}
    These two constructions are strict 2-pullbacks, and the strict 2-functor $\vpt : \TopSp\to\vUlt$ preserves these strict 2-pullbacks.
\end{cons}

\begin{defn}
    A \emph{codescent-diagram} in $\K$ is a pseudofunctor $X : \Delta_2^-\to\K$, where $\Delta_2^-$ denotes the subcategory of the cosimplicial category generated by the following maps
    \[\begin{tikzcd}
        {[2]} & {[1]} & {[0]}
        \arrow["d_1"{description}, from=1-1, to=1-2]
        \arrow["d_2"{description}, shift left=5, from=1-1, to=1-2]
        \arrow["d_0"{description}, shift right=5, from=1-1, to=1-2]
        \arrow["d_1"{description}, shift left=3, from=1-2, to=1-3]
        \arrow["d_0"{description}, shift right=3, from=1-2, to=1-3]
        \arrow["s_0"{description}, shorten <=3pt, shorten >=3pt, from=1-3, to=1-2]
    \end{tikzcd}\] that we see as a 2-category with only identity 2-arrows.
    A codescent-diagram will be thought of as the structure of an internal category: we write
    \[\begin{tikzcd}
    	{X_2} & {X_1} & {X_0}
    	\arrow["m"{description}, from=1-1, to=1-2]
    	\arrow["{\pi_2}"{description}, shift left=5, from=1-1, to=1-2]
    	\arrow["{\pi_1}"{description}, shift right=5, from=1-1, to=1-2]
    	\arrow["s"{description}, shift left=3, from=1-2, to=1-3]
    	\arrow["t"{description}, shift right=3, from=1-2, to=1-3]
    	\arrow["u"{description}, shorten <=3pt, shorten >=3pt, from=1-3, to=1-2]
    \end{tikzcd} \hspace{10pt} \text{ instead of } \hspace{10pt}  \begin{tikzcd}
    	{X([2])} & {X([1])} & {X([0])}
    	\arrow["d_1"{description}, from=1-1, to=1-2]
    	\arrow["d_2"{description}, shift left=5, from=1-1, to=1-2]
    	\arrow["d_0"{description}, shift right=5, from=1-1, to=1-2]
    	\arrow["d_1"{description}, shift left=3, from=1-2, to=1-3]
    	\arrow["d_0"{description}, shift right=3, from=1-2, to=1-3]
    	\arrow["s_0"{description}, shorten <=3pt, shorten >=3pt, from=1-3, to=1-2]
    \end{tikzcd},\]
    and such a codescent-diagram will be denoted by $(X_{\bullet})$.

\end{defn}

\begin{ex}\leavevmode
    \begin{enumerate}
        \item Any internal category (or internal groupoid) induces a codescent-diagram.
        \item Any 1-arrow $F : X_0\to \eZ$ in $\K$ gives an internal groupoid, and so a codescent-diagram, by looking at its iterated pseudokernel:
    \[\begin{tikzcd}
    	{X_0\times_{\eZ}X_0\times_{\eZ}X_0} & {X_0\times_{\eZ}X_0} & {X_0}
    	\arrow[from=1-1, to=1-2]
    	\arrow[shift left=4, from=1-1, to=1-2]
    	\arrow[shift right=4, from=1-1, to=1-2]
    	\arrow[shift left=2, shorten <=5pt, shorten >=5pt, from=1-2, to=1-1]
    	\arrow[shift right=2, shorten <=5pt, shorten >=5pt, from=1-2, to=1-1]
    	\arrow[shift right=2, from=1-2, to=1-3]
    	\arrow[shift left=2, from=1-2, to=1-3]
    	\arrow[shorten <=6pt, shorten >=6pt, from=1-3, to=1-2]
    \end{tikzcd}.\]
    We call this internal groupoid the \emph{groupoid kernel of $F$}.
    In $\TopSp$ and $\vUlt$ we can use the strict 2-pullbacks of Construction \ref{cons:strict2pb} to define the pseudokernel and then the groupoid laws are satisfied strictly, this will always be the case in the following.
    \end{enumerate}
\end{ex}
    
By taking the groupoid kernel of $F$, we attempt to give enough data above $X_0$ to recover $\eZ$. There is a reverse construction: starting from a codescent-diagram we can try to take its ‘‘quotient’’; this is the following notion of \emph{universal descent-cocone}, the correct notion of ‘‘coequalizer’’ for a codescent-diagram.

\begin{defn}\label{def:desc-cocone-cat}\leavevmode
    For $\eZ$ an object of $\K$, a \emph{descent-cocone of apex $\eZ$} over a codescent-diagram $(X_{\bullet})$ is a pair $(F,\theta)$ of a 1-arrow $F:X_0\to\eZ$ and an invertible 2-arrow $\theta : F\circ s \stackrel{\sim}\Rightarrow F\circ t$ satisfying the following two equations:
\begin{enumerate}
    \item a \emph{unit condition}
    \[\begin{tikzcd}
	&& {X_0} \\
	{X_0} & {X_1} && \eZ \\
	&& {X_0}
	\arrow["F", from=1-3, to=2-4]
	\arrow["\theta", shorten <=20pt, shorten >=18pt, Rightarrow, from=1-3, to=3-3]
	\arrow[""{name=0, anchor=center, inner sep=0}, "\id", curve={height=-18pt}, from=2-1, to=1-3]
	\arrow["u"', from=2-1, to=2-2]
	\arrow[""{name=1, anchor=center, inner sep=0}, "\id"', curve={height=18pt}, from=2-1, to=3-3]
	\arrow["s", from=2-2, to=1-3]
	\arrow["t"', from=2-2, to=3-3]
	\arrow["F"', from=3-3, to=2-4]
	\arrow["\wr"'{pos=0.2}, shorten <=5pt, shorten >=5pt, Rightarrow, from=0, to=2-2]
	\arrow["\wr"'{pos=0.6}, shorten <=5pt, shorten >=5pt, Rightarrow, from=2-2, to=1]
\end{tikzcd} \hspace{10pt} \text{ = } \hspace{10pt} \id_F, \text{ and}\]
    \item a \emph{cocycle condition}.
\[\begin{tikzcd}
	& {X_1} & {X_0} \\
	{X_2} && {X_0} & \eZ \\
	& {X_1} & {X_0}
	\arrow["s", from=1-2, to=1-3]
	\arrow["t"', from=1-2, to=2-3]
	\arrow["\wr"', shorten <=20pt, shorten >=18pt, Rightarrow, from=1-2, to=3-2]
	\arrow["\theta", shorten <=3pt, shorten >=3pt, Rightarrow, from=1-3, to=2-3]
	\arrow["F", from=1-3, to=2-4]
	\arrow["{\pi_2}", from=2-1, to=1-2]
	\arrow["{\pi_1}"', from=2-1, to=3-2]
	\arrow["F"{description}, from=2-3, to=2-4]
	\arrow["\theta", shorten <=3pt, shorten >=3pt, Rightarrow, from=2-3, to=3-3]
	\arrow["s", from=3-2, to=2-3]
	\arrow["t"', from=3-2, to=3-3]
	\arrow["F"', from=3-3, to=2-4]
\end{tikzcd} \hspace{10pt} \text{ = } \hspace{10pt} \begin{tikzcd}
	& {X_1} & {X_0} \\
	{X_2} & {X_1} && \eZ \\
	& {X_1} & {X_0}
	\arrow["s", from=1-2, to=1-3]
	\arrow["\wr"'{pos=0.4}, shorten <=3pt, shorten >=3pt, Rightarrow, from=1-2, to=2-2]
	\arrow["F", from=1-3, to=2-4]
	\arrow["\theta", shorten <=20pt, shorten >=18pt, Rightarrow, from=1-3, to=3-3]
	\arrow["{\pi_2}", from=2-1, to=1-2]
	\arrow["m"{description}, from=2-1, to=2-2]
	\arrow["{\pi_1}"', from=2-1, to=3-2]
	\arrow["s"', from=2-2, to=1-3]
	\arrow["\wr"'{pos=0.4}, shorten <=3pt, shorten >=3pt, Rightarrow, from=2-2, to=3-2]
	\arrow["t", from=2-2, to=3-3]
	\arrow["t"', from=3-2, to=3-3]
	\arrow["F"', from=3-3, to=2-4]
\end{tikzcd}\]
\end{enumerate}
    The descent-cocones of apex $\eZ$ over $(X_{\bullet})$ form a category $\Desc(X_{\bullet};\eZ)$, a morphism from $(F,\theta)$ to $(F',\theta')$ is given by a 2-arrow $\alpha : F\Rightarrow F'$ such that $(\alpha\cdot t)\circ\theta = \theta'\circ(\alpha\cdot s)$. The construction $\Desc(X_{\bullet};\eZ)$ is functorial in $\eZ$, in the sense that it induces a strict 2-functor $\Desc(X_{\bullet};-) : \K\to \Cat$.
\end{defn}
\begin{defn}
    A descent-cocone $(F,\theta)\in\Desc(X_{\bullet};\eZ)$ is \emph{universal} if it weakly represents the functor $\Desc(X_{\bullet};-) : \K\to \Cat$ above, \ie if for any object $\eY$ the functor $(F,\theta)^* : \K(\eZ,\eY)\to\Desc(X_{\bullet};\eY)$ is an equivalence of categories.
    Any 1-arrow $F : X_0\to \eZ$ forms a descent-cocone over its own groupoid kernel, we say that $F$ is an \emph{effective descent map} if this descent-cocone is universal.
\end{defn}

\begin{rmk}
    The notion of universal descent-cocone is the appropriate notion of ‘‘coequalizer’’ for codescent-diagrams. It is actually a colimit of the codescent-diagram for a suitable weight, we refer to \cite[Section 2]{CMV} for more details.
    Effective descent maps are the appropriate notion of ‘‘effective epimorphisms’’, and we will think of an effective descent map $F$ as a well-behaved quotient: the groupoid kernel of $F$ enables us to recover $F$ in a canonical way.
\end{rmk}

\begin{defn}\leavevmode
    \begin{enumerate}
        \item A \emph{topological groupoid} is a groupoid internal to $\TopSp$. 
        An \emph{equivariant sheaf} over a topological groupoid $(T_{\bullet})$ is a descent-cocone over $(\Sh(T_{\bullet}))$ with apex $\SetO$ in the 2-category $\GTop$, as explained in Definition \ref{def:desc-cocone-cat} they form a category
        \[\Sh_{eq}(T_{\bullet})\coloneqq\Desc(\Sh(T_{\bullet});\SetO).\]
        Concretely, an equivariant sheaf over $(T_{\bullet})$ consists of a sheaf over the space of objects $T_0$ together with a continuous action of the space of arrows $T_1$.
        A topos $\E$ is said to be \emph{represented} by a topological groupoid $(T_{\bullet})$ if there is a universal descent-cocone over $(T_{\bullet})$ with apex $\E$. Note that then the topos $\E$ is determined by
        \[\E \simeq \GTop(\E,\SetO) \simeq \Desc(\Sh(T_{\bullet});\SetO) \eqqcolon \Sh_{eq}(T_{\bullet}).\] 
        
        \item Similarly, an \emph{equivariant ultrasheaf} over a $\vUlt$-groupoid $(X_{\bullet})$ is a descent-cocone over $(X_{\bullet})$ with apex $\Set$ in the 2-category $\vUlt$, they form a category \[\vSh_{eq}(X_{\bullet})\coloneqq\Desc(X_{\bullet};\Set).\]
        And a v-ultracategory $\eZ$ is said to be \emph{represented} by a $\vUlt$-groupoid $(X_{\bullet})$ if there is a universal descent-cocone over $(X_{\bullet})$ with apex $\eZ$.
    \end{enumerate}
\end{defn}

\begin{rmk}\label{rmk:equivsheaves-equiv}
    We have shown in Proposition \ref{prop:0dimcase} that giving a sheaf over a topological space is the same as giving an ultrasheaf over the v-ultracategory of its points. Hence, the category of equivariant sheaves over a topological groupoid $(T_{\bullet})$ is equivalent to the category of equivariant ultrasheaves over $(\vpt(T_{\bullet}))$.
\end{rmk}

\begin{rmk}
    If $\E$ is represented by a topological groupoid $(T_{\bullet})$, then the set of points of $T_0$ induces a separating set of points of $\E$ and so $\E$ has enough points. In fact, as shown by Butz and Moerdijk in \cite{BM1}, a topos with enough points can always be represented by a topological groupoid. 
\end{rmk}

\begin{cons}\label{cons:strategyproof}
    Recall that our objective is to prove Proposition \ref{prop:small-sep-case}. We know from Proposition \ref{prop:0dimcase} that the result holds for topological spaces, our strategy will thus consist of representing the topos by a topological groupoid to reduce to the 0-dimensional case. 

    Suppose that $\E$ is represented by $(T_{\bullet})$, \ie there is a universal descent-cocone $(\pi,\theta)$ above $(\Sh(T_{\bullet}))$.
    \[\begin{tikzcd}[column sep=small]
    	{\Sh(T_1\underset{T_0}{\times}T_1)} & {\Sh(T_1)} & {\Sh(T_0)} & \E
    	\arrow[from=1-1, to=1-2]
    	\arrow[shift right=2, from=1-2, to=1-3]
    	\arrow[shift left=2, from=1-2, to=1-3]
    	\arrow[shorten <=3pt, shorten >=3pt, from=1-3, to=1-2]
    	\arrow["\pi", curve={height=-20pt}, dashed, from=1-3, to=1-4]
    \end{tikzcd} \hspace{25pt} \begin{tikzcd}[column sep=small]
    	{\Sh(T_1)} & {\Sh(T_0)} \\
    	{\Sh(T_0)} & \E
    	\arrow[from=1-1, to=1-2]
    	\arrow[from=1-1, to=2-1]
    	\arrow["\theta","\sim"', shorten <=9pt, shorten >=9pt, Rightarrow, from=2-1, to=1-2]
    	\arrow["\pi"', from=2-1, to=2-2]
    	\arrow["\pi", from=1-2, to=2-2]
    \end{tikzcd}\]
    Going through the functor $\vpt$ we get a descent-cocone $(\vpt(\pi),\vpt(\theta))$ over $(\vpt(T_{\bullet}))$.
    \[\begin{tikzcd}[column sep=small]
    	{\vpt(T_1)\underset{\vpt(T_0)}\times\vpt(T_1)} & {\vpt(T_1)} & {\vpt(T_0)} & {\vpt(\E;T_0)}
    	\arrow[from=1-1, to=1-2]
    	\arrow[shift right=2, from=1-2, to=1-3]
    	\arrow[shift left=2, from=1-2, to=1-3]
    	\arrow[shorten <=3pt, shorten >=3pt, from=1-3, to=1-2]
    	\arrow["\vpt(\pi)", curve={height=-20pt}, dashed, from=1-3, to=1-4]
    \end{tikzcd} \hspace{25pt} \begin{tikzcd}[column sep=small]
    	{\vpt(T_1)} & {\vpt(T_0)} \\
    	{\vpt(T_0)} & {\vpt(\E;T_0)}
    	\arrow[from=1-1, to=1-2]
    	\arrow[from=1-1, to=2-1]
    	\arrow["\vpt(\theta)","\sim"', shorten <=9pt, shorten >=9pt, Rightarrow, from=2-1, to=1-2]
    	\arrow["\vpt(\pi)"', from=2-1, to=2-2]
    	\arrow["\vpt(\pi)", from=1-2, to=2-2]
    \end{tikzcd}\]
\end{cons}
    
\begin{question}\label{question}
    The cocone $(\pi,\theta)$ is universal by hypothesis, is the descent-cocone $(\vpt(\pi),\vpt(\theta))$ also universal? \ie is $\vpt(\E;T_0)$ represented by the $\vUlt$-groupoid $(\vpt(T_{\bullet}))$?
\end{question}

If it were the case, we would get the following sequence of equivalences
\[\vSh(\vpt(\E;T_0))\simeq\vSh_{eq}(\vpt(T_{\bullet}))\simeq\Sh_{eq}(T_{\bullet})\simeq\E,\]
thus, to prove Proposition \ref{prop:small-sep-case}, it is enough to construct a topological groupoid representing $\E$ that answers positively to Question \ref{question}.
As we expect the functor $\vpt$ to be a right adjoint (see Corollary \ref{cor:final-adj}), we cannot expect it to preserve colimits, and so we cannot expect \ref{question} to hold for an arbitrary topological groupoid. However, we will construct in the next section a particular groupoid representing $\E$ that satisfies \ref{question}; the main ingredient will be the following technical result, which gives a sufficient condition for a functor of v-ultracategories to be effective descent.

\begin{prop}\label{prop:desc-vUlt}
    Let $\pi : X_0\to\eZ$ be a functor of virtual ultracategories, surjective on objects and satisfying the following lifting property: any ultraarrow $f : \pi(x)\ultrato (b_s)_{s:\mu}$ can be lifted to an ultraarrow $x\ultrato (y_s)_{s:\mu}$ of the same type. Then $\pi$ is an effective descent functor.
\end{prop}
\begin{proof}
    We denote by $(X_1,\theta)$ the 2-pullback of $\pi$ along itself so that $(\pi,\theta)$ forms a descent-cocone above $(X_{\bullet})$ and we want to show that this descent-cocone is universal. Let us fix $\eY$ another v-ultracategory; we want to show that the functor $(\pi,\theta)^* : \vUlt(\eZ,\eY)\to\Desc(X_{\bullet};\eY)$ is an equivalence of categories.

    The faithfulness comes from the surjectivity on objects of $\pi$: let $F,G:\eZ\to\eY$ and $\alpha,\beta:F\Rightarrow G$ such that $\alpha$ and $\beta$ have the same image by $(\pi,\theta)^*$, so $\alpha\pi = \beta\pi$, and we conclude that $\alpha = \beta$ using that $\pi$ is surjective on objects.

    The proofs of fullness and essential surjectivity are a bit more involved. Using the hypotheses, we can fix the following liftings:
    for each object $a\in\eZ$ there is $\li(a)\in X_0$ such that $\pi(\li(a)) = a$ and for each ultraarrow $f:a\ultrato(b_s)_{s:\mu}$ in $\eZ$ and $x$ such that $\pi(x) = a$, there is $\li(f;x) : x\ultrato (\lii(f;x)_s)_{s:\mu}$ such that $\pi(\li(f;x)) = f$.

    We first show the fullness. Let $F,G : \eZ\to\eY$ and $\alpha : F\circ\pi \to G\circ\pi$ be a morphism of descent-cocones from $(\pi,\theta)^*(F)$ to $(\pi,\theta)^*(G)$, \ie 
    \[(\alpha \cdot t) \circ (F \cdot \theta) = (G \cdot \theta)\circ(\alpha \cdot s).\]
    A candidate for a preimage of $\alpha$ is given by $\gamma : F\Rightarrow G$ defined by $\gamma_a \coloneqq \alpha_{\li(a)}$; we have to check the naturality of $\gamma$.
    Let $f : a\ultrato(b_s)_{s:\mu}$ be an ultraarrow in $\eZ$, the following diagram shows the naturality of $\gamma$ at $f$.
    \[\begin{tikzcd}
	{F(\pi(\li(a)))} && {G(\pi(\li(a)))} \\
	{F(\pi(\lii(f;\li(a))_s))} && {G(\pi(\lii(f;\li(a))_s))} \\
	{F(\pi(\li(b_s)))} && {G(\pi(\li(b_s)))}
	\arrow["{\alpha_{\li(a)}}", from=1-1, to=1-3]
	\arrow["{F(\pi(\li(f;x)))}", -to reversed, from=1-1, to=2-1]
	\arrow["{F(f)}"', -to reversed, curve={height=70pt}, from=1-1, to=3-1]
	\arrow["{G(\pi(\li(f;x)))}"', -to reversed, from=1-3, to=2-3]
	\arrow["{G(f)}", -to reversed, curve={height=-70pt}, from=1-3, to=3-3]
	\arrow["{(\alpha_{\lii(f;\li(a))_s})}"', from=2-1, to=2-3]
	\arrow[equals, from=2-1, to=3-1]
	\arrow[equals, from=2-3, to=3-3]
	\arrow["{(\alpha_{\li(b_s)})}"', from=3-1, to=3-3]
\end{tikzcd}\]
    The upper square is a naturality square of $\alpha$, and the lower square is the equation $(\alpha \cdot t) \circ (F \cdot \theta) = (G \cdot \theta)\circ(\alpha \cdot s)$ applied to $(\id_{b_s},\lii(f;\li(a))_s,\li(b_s))\in X_1$.

    We now show the essential surjectivity. Let $(H,\eta)$ be a descent-cocone over $(X_{\bullet})$ of apex $\eY$, we want to construct a preimage $F : \eZ\to\eY$.
    We define $F : \eZ\to\eY$ as follows:
    \begin{align*}
        & F(a)\coloneqq H(\li(a)) \text{ , for $a$ object of $\eZ$,}
        \\ & F(f) : H(\li(a)) \stackrel{H(\li(f;\li(a)))}\longultrato H(\lii(f;\li(a))_s)_{s:\mu} \stackrel{(\eta)_{\mu}}\longrightarrow H(\li(b_s))_{s:\mu} \text{ , for $f:a\ultrato(b_s)$ in $\eZ$}.
    \end{align*}

    Only remains to show that $F$ defines a functor (it is not obvious as the lifts are chosen arbitrarily) and that it is a preimage of $(H,\eta)$. There are a lot of coherences to check, but everything flows fluently.

    \begin{itemize}
        \item First, notice that for any $f : a\ultrato(b_s)$ with two lifts $g : x\ultrato(y_s)$, $g' : x'\ultrato(y'_s)$, the following square commutes.
        \begin{equation}\tag{$\ast$}\label{ast}
            \begin{tikzcd}[ampersand replacement=\&]
            	{H(x)} \& {H(y_s)} \\
            	{H(x')} \& {H(y'_s)}
            	\arrow["{H(g)}", -to reversed, from=1-1, to=1-2]
            	\arrow["\eta"', from=1-1, to=2-1]
            	\arrow["(\eta)", from=1-2, to=2-2]
            	\arrow["{H(g')}", -to reversed, from=2-1, to=2-2]
            \end{tikzcd}
        \end{equation}
        It is the naturality square of $\eta$ for the ultraarrow $(\id_a,x,x')\ultrato(\id_{b_s},y_s,y'_s)_{s:\mu}$ of $X_1$ given by $(g,g')$.

        \item \ul{$F(\id_a) = \id_{F(a)}$:} for $a$ an object of $\eZ$, the following square
        \[\begin{tikzcd}
        	{H(\li(a))} && {H(\lii(\id_a;\li(a))} \\
        	{H(\li(a))} && {H(\li(a))}
        	\arrow["{H(\li(\id_a;\li(a)))}", from=1-1, to=1-3]
        	\arrow["\eta "', equals, from=1-1, to=2-1]
        	\arrow["\eta", from=1-3, to=2-3]
        	\arrow["{\id_{H(\li(a))}}", from=2-1, to=2-3]
        \end{tikzcd}\]
        (where the left equality is $\eta$ by the unit condition) commutes by \eqref{ast}.

        \item \ul{$F((g_s)\circ f) = (F(g_s))\circ F(f)$:} for $f : a\ultrato(b_s)$ and $(g_s : b_s \ultrato (c_{s,t}))$ ultraarrows in $\eZ$, we need to show the commutation of the following diagram.
        \[\begin{tikzcd}
        	{H(\li(a))} &&&&& {H(\lii(gf;\li(a))_{s,t})} \\
        	{H(\li(a))} && {H(\lii(f;\li(a))_s)} &&& {H(\lii(g_s;\lii(f;\li(a)))_{t})} \\
        	&& {H(\li(b_s))} &&& {H(\lii(g_s;\li(b_s))_t)} \\
        	&&&&& {H(\li(c_{s,t}))}
        	\arrow["{H(\li(gf;\li(a)))}", -to reversed, from=1-1, to=1-6]
        	\arrow["\eta"', equals, dashed, from=1-1, to=2-1]
        	\arrow["\eta", dashed, from=1-6, to=2-6]
        	\arrow["\eta", shift left=7, curve={height=-50pt}, from=1-6, to=4-6]
        	\arrow["{H(\li(f;\li(a)))}", -to reversed, from=2-1, to=2-3]
        	\arrow["{H(\li(g_s;\lii(f;\li(a))))}", -to reversed, dashed, from=2-3, to=2-6]
        	\arrow["\eta"', from=2-3, to=3-3]
        	\arrow["\eta", dashed, from=2-6, to=3-6]
        	\arrow["{H(\li(g_s;\li(b_s)))}", -to reversed, from=3-3, to=3-6]
        	\arrow["\eta", from=3-6, to=4-6]
        \end{tikzcd}\]
        The right hand side commutes by the cocycle condition, the left equality is $\eta$ by the unit condition and the middle and upper squares commute by \eqref{ast}.
    \end{itemize}
    
    This shows the functoriality of $F: \eZ\to\eY$.
    To show that $F$ is a preimage of $(H,\eta)$ we construct a natural isomorphism $\Phi : H \stackrel{\sim}\Rightarrow F\circ\pi$ by
    \[\Phi_x : H(x) \stackrel{\eta}\longrightarrow H(\li(\pi(x))) = F(\pi(x)) \text{ , for $x$ an object of $X_0$}.\]
    
    \begin{itemize}
        \item \ul{$\Phi$ is natural:} for any $f : x\ultrato(y_s)$ in $X_0$, we need to show the following commutation.
        \[\begin{tikzcd}[ampersand replacement=\&]
        	{H(x)} \&\&\& {H(y_s)} \\
        	{H(\li(\pi(x)))} \&\& {H(\lii(f;\li(\pi(x)))_s)} \& {H(\li(\pi(y_s)))} \\
        	{F(\pi(x))} \&\&\& {F(\pi(y_s))}
        	\arrow["{H(f)}", -to reversed, from=1-1, to=1-4]
        	\arrow["\eta", from=1-1, to=2-1]
        	\arrow["\Phi_x"', curve={height=50pt}, from=1-1, to=3-1]
        	\arrow["\eta", from=1-4, to=2-4]
        	\arrow["(\Phi_{y_s})", curve={height=-50pt}, from=1-4, to=3-4]
        	\arrow["{H(\li(f;\li(\pi(x))))}", -to reversed, from=2-1, to=2-3]
        	\arrow["\eta", from=2-3, to=2-4]
        	\arrow[equals, from=3-1, to=2-1]
        	\arrow["{F(f)}", -to reversed, from=3-1, to=3-4]
        	\arrow[equals, from=3-4, to=2-4]
        \end{tikzcd}\]
        The lower square is the definition of $F$, and the upper square commutes by \eqref{ast}.

        \item \ul{$\Phi$ induces an isomorphism between $(F\pi,F\theta)$ and $(H,\eta)$:}
        for any $(v,x,y)\in X_1$ (\ie $x,y\in X_0$ and $v:\pi(x)\to\pi(y)$ isomorphism), we need to show the following commutation.
        \[\begin{tikzcd}
        	{H(x)} &&& {H(y)} \\
        	{H(\li(\pi(x)))} && {H(\lii(v;\li(\pi(x))))} & {H(\li(\pi(y)))} \\
        	{F(\pi(x))} &&& {F(\pi(y))}
        	\arrow["{\eta_{(v,x,y)}}", from=1-1, to=1-4]
        	\arrow["\eta", from=1-1, to=2-1]
        	\arrow["\Phi_x"', curve={height=50pt}, from=1-1, to=3-1]
        	\arrow["\eta", from=1-4, to=2-4]
        	\arrow["\Phi_y", curve={height=-50pt}, from=1-4, to=3-4]
        	\arrow["{H(\li(v;\li(\pi(x))))}", from=2-1, to=2-3]
        	\arrow["\eta", from=2-3, to=2-4]
        	\arrow[equals, from=3-1, to=2-1]
        	\arrow["{F(v)}", from=3-1, to=3-4]
        	\arrow[equals, from=3-4, to=2-4]
        \end{tikzcd}\]
        The lower square is again the definition of $F$, and using the cocycle and the unit condition of $\eta$, the commutation of the upper square amounts to the one of 
        \[\begin{tikzcd}
        	{H(x')} && {H(y')} \\
        	{H(y')} && {H(y')}
        	\arrow["{\eta_{(v,x',y')}}", from=1-1, to=1-3]
        	\arrow["{H(f)}"', from=1-1, to=2-1]
        	\arrow[equals, from=1-3, to=2-3]
        	\arrow["{\eta_{(\id_{y'},y',y')}}", equals, from=2-1, to=2-3]
        \end{tikzcd}\]
        where we denote $x'\coloneqq \li(\pi(x))$, $y'\coloneqq \lii(v;\li(\pi(x)))$, and $f \coloneqq \li(v;\li(\pi(x)))$. And this square is the naturality square of $\eta$ for the arrow $(v,x',y')\to(\id_{y'},y',y')$ given by $(f,\id_{y'})$. 
    \end{itemize}
\end{proof}

\begin{rmk}
    The lifting hypothesis of Proposition \ref{prop:desc-vUlt} is a categorified version of Theorem \ref{Thm:étale-ucvg}\ref{cond2}. 
    It is shown in \cite{RT} that open surjections are effective descent in $\TopSp$; however, the hypotheses of Proposition \ref{prop:desc-vUlt} are a bit stronger than being an open surjection in the 0-dimensional case as we ask for the lift of ultrafilters to be of the same type. In \cite[Theorem 1.5]{RT} the authors give a full characterization of effective descent maps between topological spaces; one can hope for a full characterization of effective descent maps in $\vUlt$ in the same flavor.
\end{rmk}  

\section{The groupoid of amply indexed models}\label{sec:ample}

In this section, we construct the topological groupoid $(T^{\amp}_{\bullet})$ of \emph{amply indexed models}. This topological groupoid will represent $\E$ and the functor
\[\vpt(\pi):\vpt(T^{\amp}_0)\longrightarrow\vpt(\E;T^{\amp}_0)\]
(as in Construction \ref{cons:strategyproof}) will satisfy the hypotheses of Proposition \ref{prop:desc-vUlt}. This will imply that $\vpt(\pi)$ is effective descent, and (checking that $(\vpt(T^{\amp}_{\bullet}))$ is indeed the groupoid kernel of $\vpt(\pi)$) we will get that $(T^{\amp}_{\bullet})$ satisfies Question \ref{question} positively.

Following \cite{BM1} we give an explicit description of $(T^{\amp}_{\bullet})$ using \emph{indexed models}; however, for the groupoid to satisfy the hypotheses of Proposition \ref{prop:desc-vUlt} we need an additional condition: we ask for the indexings to be \emph{ample}. We then use a theorem of Wrigley \cite[Proposition 73]{Josh} to conclude that this ample condition is not too strong, \ie that the groupoid of amply indexed models represents our topos $\E$.

We first recall how to construct a topological groupoid representing a topos $\E$ and how indexed models appear. Our presentation follows closely Butz--Moerdijk \cite{BM1,BM2}.

\begin{cons}\label{cons:XMk}
We want to construct a topological groupoid $(T_{\bullet})$ representing $\E$. The points of $T_0$ will induce points of $\E$, so (as $T_0$ is small) we have to choose a small set of points of $\E$, and we want this set of points to be separating to have some hope to be able to reconstruct $\E$. 

\begin{nnota}
    We fix $X$ a small separating set of points of $\E$; we think of the points of $X$ as a conservative set of models. 
\end{nnota}

We look now at two points $a$ and $b$ of $T_0$ that are sent onto two points $p$ and $q$ of $X$. If $a\cleq b$ in $T_0$ the ultraconvergence induces a map of points $p\to q$, but in general, there are plenty of such maps, we thus need a way to add some information on $a$ and $b$ to be able to trace out the map $p\to q$ to which $a\cleq b$ will be sent.
The idea is to add an indexing: a point $a$ will be a pair $(p,\alpha)$ where $\alpha$ is an indexing of the model pointed by $p$, so that if $a = (p,\alpha)\cleq b=(q,\beta)$, the two indexings $\alpha$ and $\beta$ will suffice to determine to which map $p\to q$ the ultraconvergence $a\cleq b$ is sent.

\begin{nnota}
    We fix an $M\in\E$ such that the subobjects of the $M^n$ generate $\E$ all together (for example, the coproduct of the objects of a small site presentation of $\E$); we think of $M_p$ as the underlying set of the model pointed by $p$.
\end{nnota}
\begin{nnota}
    We fix $\kappa$ a large enough infinite cardinal such that $\kappa\geq|M_p|$ for all $p$ in $X$; $\kappa$ will play the role of the indexing set.
\end{nnota}

The data of such $\E$, $X$, $M$, and $\kappa$ will be fixed until the end of this section.
\end{cons}

\begin{defn}\leavevmode
    \begin{enumerate}
        \item An \emph{indexing} of a point $p\in X$ is a partial surjection $\alpha : \kappa \paronto M_p$. A point $p$ of $X$ together with such an indexing is called an \emph{indexed model}.
        
        \item An \emph{ultramorphism of indexed models} from $(p,\alpha)$ to $(q_s,\beta_s)_{s:\mu}$ is an ultraarrow $f:p\ultrato(q_s)$ in $\vpt(\E;X)$ such that 
        \[f_M(\alpha(i)) = (\beta_s(i))_{s:\mu} \text{ for all } i\in\dom(\alpha),\]
        this condition can be expressed as the existence of a (necessarily unique) lift
\[\begin{tikzcd}
	\kappa & \kappa^\mu \\
	{\dom(\alpha)} & {\int_{s:\mu}\dom(\beta_s)} \\
	{M_p} & {\int_{s:\mu}M_{q_s}}
	\arrow[from=1-1, to=1-2]
	\arrow["\subseteq"{marking, allow upside down}, draw=none, from=2-1, to=1-1]
	\arrow[dashed, from=2-1, to=2-2]
	\arrow["\alpha"', two heads, from=2-1, to=3-1]
	\arrow["\subseteq"{marking, allow upside down}, draw=none, from=2-2, to=1-2]
	\arrow["(\beta_s)", two heads, from=2-2, to=3-2]
	\arrow["{f_M}"', from=3-1, to=3-2]
\end{tikzcd}\]
    and we say that this ultramorphism of indexed models is \emph{witnessed} by $f$.
        
        \item An \emph{indexing} $\alpha : \kappa \paronto M_p$ is said to be \emph{ample} if $|\kappa\setminus\dom(\alpha)|=\kappa$. An \emph{amply indexed model} is an indexed model $(p,\alpha)$ such that $\alpha$ is ample.
    \end{enumerate}
\end{defn}

The reason we introduced the ample condition is the following technical result.
\begin{lem}\label{lem:ample}\leavevmode
    \begin{enumerate}
        \item All points in $X$ have an ample indexing.
        \item For $(p,\alpha)$ an amply indexed model, and $f:p\ultrato(q_s)_{s:\mu}$ an ultraarrow in $\vpt(\E;X)$, we can find ample indexings $(\beta_s)_{s:\mu}$ such that $f$ witnesses an ultramorphism of indexed models from $(p,\alpha)$ to $(q_s,\beta_s)_{s:\mu}$.
    \end{enumerate}
\end{lem}
\begin{proof}\leavevmode
    \begin{enumerate}
        \item Let $p$ be a point in $X$. By assumption on $\kappa$, $\kappa \geq |M_p|$, and as $\kappa$ is infinite $|\kappa\sqcup\kappa| = \kappa$, hence 
        \[\kappa \simeq\kappa\sqcup\kappa\supset\kappa\twoheadrightarrow M_p\] 
        gives an ample indexing of $p$.
        \item We choose a support set $S$ of $\mu$ such that $\prod_{s:S}M_{q_s}\to\int_{s:\mu}M_{q_s}$ is surjective. So that we can lift $f_M\circ\alpha : \dom(\alpha)\to\int_{s:\mu}M_{q_s}$ to a map into the product $\prod_{s:S}M_{q_s}$, getting a family $(\alpha'_s: \dom(\alpha) \to M_{q_s})_{s:S}$. As $\alpha$ is ample, $|\kappa\setminus\dom(\alpha)| = \kappa$ and we extend each $\alpha'_s : \dom(\alpha)\to M_{q_s}$ to an ample indexing $\beta_s : \kappa\paronto M_{q_s}$.
\[\begin{tikzcd}
	\kappa && {\kappa^S} & {\kappa^{\mu}} \\
	{\dom(\alpha)} & {\prod \dom(\alpha)} & {\prod\dom(\beta_s)} & {\int\dom(\beta_s)} \\
	\\
	{M_p} && {\prod M_{q_s}} & {\int M_{q_s}}
	\arrow[hook, from=1-1, to=1-3]
	\arrow[two heads, from=1-3, to=1-4]
	\arrow["\subset"{marking, allow upside down}, draw=none, from=2-1, to=1-1]
	\arrow[hook, from=2-1, to=2-2]
	\arrow["\alpha"', from=2-1, to=4-1]
	\arrow["{(\alpha'_s)}"{description}, curve={height=12pt}, dashed, from=2-1, to=4-3]
	\arrow["\subseteq"{marking, allow upside down}, draw=none, from=2-2, to=2-3]
	\arrow["{\prod\alpha'_s}"{description}, dashed, from=2-2, to=4-3]
	\arrow["\subset"{marking, allow upside down}, draw=none, from=2-3, to=1-3]
	\arrow[two heads, from=2-3, to=2-4]
	\arrow["{\prod\beta_s}"{description}, dashed, two heads, from=2-3, to=4-3]
	\arrow["\subset"{marking, allow upside down}, draw=none, from=2-4, to=1-4]
	\arrow["{\int\beta_s}"{description}, dashed, two heads, from=2-4, to=4-4]
	\arrow["{f_M}"', curve={height=18pt}, from=4-1, to=4-4]
	\arrow[two heads, from=4-3, to=4-4]
\end{tikzcd}\]
    \end{enumerate}
\end{proof}

Amply indexed models form a topological space, it constitutes the space of objects $T_0^{\amp}$ of our topological groupoid representing $\E$.

\begin{cons}
    The points of $T_0^{\amp}$ are amply indexed models $(p,\alpha)$. The topology is the logical topology, \ie the topology is given by the following basic opens:
    \begin{align*}
        U_{\underline{i},A} \coloneqq \{(p,\alpha) : \alpha(\underline{i})\in A_p\} \text{, with } \underline{i} \in \kappa^n \text{ and } A\subseteq M^n.
    \end{align*}
    From a logical point of view, $A$ is seen as a formula with $n$ parameters so that $U_{\underline{i},A}$ is the set of amply indexed models satisfying this formula, the parameters being specified by the indexing.
    
    Moreover, there is a geometric morphism $\pi : \Sh(T_0^{\amp})\to\E$, for $E\in\E$ the sheaf $\pi^*(E)\to T_0^{\amp}$ is defined by $\pi^*(E) \coloneqq \bigsqcup_{(p,\alpha):T_0^{\amp}}E_p$ with the topology given by the following basic opens:
        \begin{align*}
            W_{\underline{i},A,h} \coloneqq \bigsqcup_{(p,\alpha):U_{\underline{i},A}}\{h(\alpha(\underline{i})))\} \text{, with } \underline{i} \in \kappa^n , A\subseteq M^n \text{ and } h : A\to E.
        \end{align*}
    One can check easily that these are bases for topologies and that $\pi$ is a geometric morphism, we refer the reader to \cite{BM1} for more details. 
\end{cons}

\begin{lem}\label{lem:T0ample}
    The virtual ultracategory $\vpt(T_0^{\amp})$ has as objects the amply indexed models and as ultraarrows the ultramorphisms of indexed models.
    Moreover, the functor $\vpt(\pi) : \vpt(T_0^{\amp})\to\vpt(\E;X)$ maps $(p,\alpha)$ to $p$ and an ultramorphism to the $f : p\ultrato(q_s)$ witnessing it.
\end{lem}
\begin{proof}
    First notice that there can be at most one ultramorphism of amply indexed models from $(p,\alpha)$ to $(q_s,\beta_s)$: let $f$ be a witness of such an ultramorphism, then as $\alpha$ is surjective the value of $f_M$ is determined, and so by the generating property of $M$ the ultraarrow $f$ is uniquely determined.

    Let $f$ be a witness of an ultramorphism of amply indexed models from $(p,\alpha)$ to $(q_s,\beta_s)$, we show that $(p,\alpha)\cleq(q_s,\beta_s)$ in $T_0^{\amp}$. Let $U_{\underline{i},A}$ be a basic open containing $(p,\alpha)$, \ie $\alpha(\underline{i})\in A_p$, then by functoriality $(\beta_s(i)) = f_M(\alpha(i)) \in \int_{s:\mu}A_{q_s}$ and so $(q_s,\beta_s)\in U_{\underline{i},A}$ for $\mu$-all $s$.

    Conversely, suppose that $(p,\alpha)\cleq(q_s,\beta_s)$, this induces an ultraarrow in $\vpt(\Sh(T_0^{\amp}))$ and by $\pi$ an ultraarrow $f : p\ultrato(q_s)$. We will show that this $f$ witnesses an ultramorphism from $(p,\alpha)$ to $(q_s,\beta_s)$. We look at the action of the ultraconvergence on the sheaf $\pi^*(M)\to T_0^{\amp}$; for any $\alpha(i)\in M_p$, the open $W_{i,M,\id_M}$ witnesses the local homeomorphism property, and so $f_M(\alpha(i)) = (\beta_s(i))$ as wanted.
\end{proof}

\begin{cor}\label{cor:T0ample-desc}
    The functor $\vpt(\pi) : \vpt(T_0^{\amp})\to\vpt(\E;X)$ is an effective descent functor.
\end{cor}
\begin{proof}
    Lemma \ref{lem:ample} together with Lemma \ref{lem:T0ample} shows that $\vpt(\pi)$ satisfies the hypotheses of Proposition \ref{prop:desc-vUlt}.
\end{proof}

The same point of $X$ appears from several different points of $T_0^{\amp}$ as instantiations of the same model with different indexings. We add a groupoidal structure $T^{\amp}_1$ on top of $T_0^{\amp}$ to identify points of $T_0^{\amp}$ representing the same model. Again, the construction of $T_1^{\amp}$ follows \cite{BM1}. 

\begin{cons}\label{cons:T1}
    The points of $T_1^{\amp}$ are triplets $(\theta,(p,\alpha),(q,\beta))$ where $(p,\alpha)$ and $(q,\beta)$ are amply indexed models, and $\theta$ is an isomorphism between $p$ and $q$ in the category of points of $\E$ (note that no conditions on the indexings are required on $\theta$). The topology is the logical topology, \ie the topology given by the following basic opens:
    \begin{align*}
        V_{\underline{i},A,\underline{j},B} \coloneqq \{(\theta,(p,\alpha),(q,\beta)) : \alpha(\underline{i})\in A_p,\beta(\underline{j})\in B_q,\theta(\alpha(\underline{i})) = \beta(\underline{j})\} \\  \text{, with } \underline{i},\underline{j} \in \kappa^n \text{ and } A,B\subseteq M^n.
    \end{align*}
    There is a natural topological groupoidal structure $(T^{\amp}_{\bullet})$.
    Moreover, noticing that for any $E\in\E$ the sheaf $\pi^*(E)\to T_0^{\amp}$ is naturally an equivariant sheaf over $(T^{\amp}_{\bullet})$, we can define a natural isomorphism
    \[\begin{tikzcd}
	    {\Sh(T_1^{\amp})} & {\Sh(T_0^{\amp})} \\
    	{\Sh(T_0^{\amp})} & {\E}
        \arrow[from=1-1, to=1-2]
    	\arrow[from=1-1, to=2-1]
    	\arrow["\pi", from=1-2, to=2-2]
    	\arrow["\sim"', "\theta", shorten <=10pt, shorten >=10pt, Rightarrow, from=2-1, to=1-2]
    	\arrow["\pi"', from=2-1, to=2-2]
    \end{tikzcd}\]
    and we obtain a descent-cocone $(\pi,\theta)$ over $(\Sh(T^{\amp}_{\bullet}))$. 
\end{cons}

\begin{lem}\label{lem:T1ample}
As in Construction \ref{cons:strategyproof}, $(\pi,\theta)$ induces a descent-cocone $(\vpt(\pi),\vpt(\theta))$ over $\vpt(T^{\amp}_{\bullet})$.
    \[\begin{tikzcd}
	{\vpt(T^{\amp}_1)} & {\vpt(T^{\amp}_0)} \\
	{\vpt(T^{\amp}_0)} & {\vpt(\E;X)}
	\arrow[from=1-1, to=1-2]
	\arrow[from=1-1, to=2-1]
	\arrow["\vpt(\pi)", from=1-2, to=2-2]
	\arrow["\sim"', "\vpt(\theta)", shorten <=10pt, shorten >=10pt, Rightarrow, from=2-1, to=1-2]
	\arrow["\vpt(\pi)"', from=2-1, to=2-2]
\end{tikzcd}\]
    The square exhibits $\vpt(T^{\amp}_1)$ as the groupoid kernel of $\vpt(\pi)$.
\end{lem}
\begin{proof}
    We describe $\vpt(T^{\amp}_1)$ and show that we get the 2-pullback of $\vpt(\pi)$ along itself given by Construction \ref{cons:strict2pb}. 
    It is clear at the level of objects: objects of $\vpt(T^{\amp}_1)$ are given by $(\theta,a,b)$ with $a$ and $b$ objects of $\vpt(T^{\amp}_0)$ and $\theta$ an isomorphism between $\vpt(\pi)(a)$ and $\vpt(\pi)(b)$ in $\vpt(\E;X)$.
    
    We fix $v \coloneqq (\theta,(p,\alpha),(q,\beta))$ and $(v_s \coloneqq (\theta_s,(p_s,\alpha_s),(q_s,\beta_s)))_{s:\mu}$ points of $T^{\amp}_1$. 
    An ultraarrow from $v$ to $(v_s)_{s:\mu}$ in the 2-pullback amounts, by Lemma \ref{lem:T0ample}, to ultraconvergences $(p,\alpha)\cleq(p_s,\alpha_s)$ (resp. $(q,\beta)\cleq(q_s,\beta_s)$) witnessed by some $f : p\ultrato(p_s)$ (resp. $g : q\ultrato (q_s)$) such that $(\theta_s)\circ f = g \circ\theta$; hence there cannot be more than one ultraarrow from $v$ to $(v_s)_{s:\mu}$ in the 2-pullback. 

    Suppose that there is such an ultraarrow, we show that $v\cleq(v_s)_{s:\mu}$ in $T^{\amp}_1$.
    Let $V_{\underline{i},A,\underline{j},B}$ be a basic open containing $v$, so $\alpha(\underline{i})\in A_p$ and $\beta(\underline{j})\in B_q$. Using Lemma \ref{lem:T0ample} we get $\alpha_s(\underline{i})\in A_{p_s}$ and $\beta_s(\underline{j})\in B_{q_s}$ for $\mu$-all $s$, and $f_M(\alpha(\underline{i})) = (\alpha_s(\underline{i}))$, $g_M(\beta(\underline{j})) = (\beta_s(\underline{j}))$. The following sequence of equalities
    \[(\theta_s(\alpha_s(\underline{i}))) = (\theta_s(f_M(\alpha(\underline{i})))) = g_M(\theta(\alpha(\underline{i}))) = g_M(\beta(\underline{j})) = (\beta_s(\underline{j}))\]
    shows that $(v_s)_{s:\mu}$ is in $V_{\underline{i},A,\underline{j},B}$.

    Conversely, suppose that $v\cleq(v_s)_{s:\mu}$ in $T^{\amp}_1$, we have $(p,\alpha)\cleq(p_s,\alpha_s)$ (resp. $(q,\beta)\cleq(q_s,\beta_s)$) witnessed by some $f : p\ultrato(p_s)$ (resp. $g:q\ultrato(q_s)$). It only remains to show that $(\theta_s)\circ f = g \circ\theta$, it is enough to show the equality on $M$; let $\alpha(i)\in M_p$ and $j$ such that $\beta(j) = \theta_M(\alpha(i))$, then $v\in V_{i,M,j,M}$ and so $(v_s)\in V_{i,M,j,M}$, this yields
    \[\theta_s(f_M(\alpha(i))) = \theta_s(\alpha_s(i)) = \beta_s(j) = g_M(\beta(j)) = g_M(\theta(\alpha(i))),\]
    showing $(\theta_s)\circ f = g \circ\theta$ on $M$ as wanted.
\end{proof}

\begin{cor}\label{cor:T1ample-desc}
    The descent-cocone $(\vpt(\pi),\vpt(\theta))$ is universal.
\end{cor}
\begin{proof}
    Follows from Corollary \ref{cor:T0ample-desc} and Lemma \ref{lem:T1ample}.
\end{proof}

The last missing piece is to show that the groupoid $(T^{\amp}_{\bullet})$ of amply indexed models represents $\E$.
As announced, this will follow from \cite[Proposition 73]{Josh}. Wrigley's theorem gives a very general condition for a groupoid of indexed models to represent the topos. Before stating the theorem, let us note that one can generalize our construction to other notions of indexings: if for each point $p\in X$ we give $\In(p)$ a subset of indexings $\kappa \paronto M_p$ of $p$, we can construct a topological groupoid $(T^{\In}_{\bullet})$ by an analogous construction to the above; for example, by taking $\In(p)$ to be the set of ample indexings of $p$ we get back our topological groupoid $(T^{\amp}_{\bullet})$ of amply indexed models.

\begin{Thm}[\cite{Josh}]
    Suppose that all $\In(p)$ are nonempty and that for each $\alpha : \kappa \paronto M_p$ in $\In(p)$ and for each partial surjection $f : \kappa\paronto\kappa$ with $\dom(f)$ a cofinite subset of $\kappa$ and all fibers of $f$ finite, the indexing $\alpha\circ f : \kappa\paronto M_p$ is again in $\In(p)$.
    Then $(T^{\In}_{\bullet})$ represents $\E$, \ie the descent-cocone $(\pi,\theta)$ associated with $(T^{\In}_{\bullet})$ is universal.
\end{Thm}

\begin{cor}\label{cor:ample-represent}
    The descent-cocone $(\pi,\theta)$ is universal.
\end{cor}
\begin{proof}
    We apply the theorem above. Let $\alpha : \kappa\paronto M_p$ be an ample indexing, then for any partial surjection $f : \kappa\paronto\kappa$ the composite $\alpha' \coloneqq \alpha\circ f : \kappa\paronto M_p$ is also an ample indexing; indeed $f$ induces a partial surjection $\kappa\setminus\dom(\alpha')\paronto\kappa\setminus\dom(\alpha)$ showing that $|\kappa\setminus\dom(\alpha')|\geq|\kappa\setminus\dom(\alpha)|$.
\end{proof}

Putting everything together, we get the following.

\begin{cor}\label{cor:ample-final}
    The functor $\ev : \E \to \vSh(\vpt(\E;X))$ is an equivalence.
\end{cor}
\begin{proof}
    The following diagram commutes:
    \[\begin{tikzcd}
    	\E & {\vSh(\vpt(\E;X))} \\
    	{\GTop(\E,\SetO)} & {\vUlt(\vpt(\E;X),\Set)} \\
     	{\Desc(\Sh(T^{\amp}_{\bullet});\SetO)} & {\Desc(\vpt(T^{\amp}_{\bullet});\Set).} 
    	\arrow["\ev", dashed, from=1-1, to=1-2]
    	\arrow["{\simeq}"{marking, allow upside down}, draw=none, from=1-1, to=2-1]
    	\arrow["{=}"{marking, allow upside down}, draw=none, from=1-2, to=2-2]
    	\arrow["{\vpt}", from=2-1, to=2-2]
    	\arrow["{(\pi{,}\theta)^*}"', from=2-1, to=3-1]
    	\arrow["{(\vpt(\pi){,}\vpt(\theta))^*}", from=2-2, to=3-2]
    	\arrow["{\vpt}", from=3-1, to=3-2]
    \end{tikzcd}\]
    The left arrow is an equivalence by Corollary \ref{cor:ample-represent}, the right one by Corollary \ref{cor:T1ample-desc}, and the bottom one by the 0-dimensional case (see Remark \ref{rmk:equivsheaves-equiv}).
    Thus the top arrow is also an equivalence.
\end{proof}

We now conclude this section by showing Proposition \ref{prop:small-sep-case}. 
\begin{proof}
    For any small separating set $X$ of points of $\E$ we can choose $M$ and $\kappa$ as in Construction \ref{cons:XMk}, and then apply Corollary \ref{cor:ample-final}.
\end{proof}

We have thus proved a particular case of the desired reconstruction result: the evaluation functor $\ev : \E\to\vSh(\vpt(\E;X))$ is an equivalence when $X$ is a small separating set of points of $\E$.

\section{The reconstruction theorem}\label{sec:ccl}

In this section we generalize Proposition \ref{prop:small-sep-case} to the case where $X$ is neither small nor separating. We show that $\ev : \E\to\vSh(\vpt(\E;X))$ is the inverse image of the \emph{restriction} of the topos $\E$ to $X$; the restriction of topoi is defined as follows.

\begin{defn}
    Let $\E$ be a topos and $X\subseteq\pt(\E)$ be a subclass of points of $\E$. We define $\E_{\res{X}}$ the \emph{restriction of $\E$ to $X$} first in the case $X$ small, and then we extend the definition to any $X$. If $X$ is small, the surjection--embedding factorization of the geometric morphism $\Set^X\to\E$ gives a subtopos of $\E$
    \[\Set^X\twoheadrightarrow \E_{\res{X}}\hookrightarrow \E.\]
    For a general $X$, we define $\E_{\res{X}}$ to be the directed union of the subtopoi $(\E_{\res{A}})$ of $\E$, where $A$ ranges over all small subsets of $X$. This is well-defined as $\E$ has only a small number of subtopoi (as they are classified by Lawvere--Tierney topologies \cite[Proposition 4.15]{ToposTheory}).
\end{defn}

\begin{rmk}
    The topos $\E_{\res{X}}$ is the smallest subtopos of $\E$ containing all points of $X$, or equivalently the biggest subtopos of $\E$ for which $X$ is separating; note that $\E_{\res{X}}$ can have more points than $X$. Another construction of the restriction to all points of a topos can be found in \cite[Corollary 7.18]{ToposTheory}.
\end{rmk}

We can now state and prove the reconstruction result in its full generality.
\begin{Thm}\label{Thm:thetrueone}
    For a topos $\E$ and $X\subseteq\pt(\E)$, the category $\vSh(\vpt(\E;X))$ is equivalent to the subtopos $\E_{\res X}\hookrightarrow\E$, and the inverse image of this embedding is given by $\ev : \E \to \vSh(\vpt(\E;X))$.
\end{Thm}
\begin{proof}
    If $X$ is small, then $X$ is a small separating set of points of $\E_{\res{X}}$, and the assertion follows from Proposition \ref{prop:small-sep-case} by noticing that $\vpt((\E_{\res{X}});X) \to \vpt(\E;X)$ is an equivalence.

    For a general $X$, let us consider $\eI$ the filtered poset of small subsets $A\subseteq X$ that are separating for $\E_{\res X}$.
    The following sequence of equations
    \begin{align*}
        \vSh(\vpt(\E;X)) &= \vUlt(\vpt(\E;X),\Set)
        \\& \simeq \vUlt(\colim_{A}\vpt(\E;A),\Set) 
        \\& \simeq \lim_{A}\vUlt(\vpt(\E;A),\Set)
        \\& = \lim_{A}\vSh(\vpt(\E;A))
    \end{align*}
    shows that $\vSh(\vpt(\E;X))$ is the pseudolimit of $\vSh(\vpt(\E;-)) : \eI^{\op}\to\CAT$.
    By the first case, the diagram $\vSh(\vpt(\E;-)) : \eI^{\op}\to\CAT$ is actually constant to $\E_{\res X}$, and so $\vSh(\vpt(\E;X))\simeq\E_{\res X}$.

    For the last assertion, it is enough to notice that for any such $A$, the functor $\ev : \E \to \vSh(\vpt(\E;X))$ postcomposed with the restriction $\vSh(\vpt(\E;X)) \to \vSh(\vpt(\E;A))$ gives the evaluation $\ev : \E \to \vSh(\vpt(\E;A))$.
\end{proof}

\begin{rmk}\label{rmk:get-Lurie-back}
    For $\E$ a coherent topos, Deligne's completeness theorem ensures that $\E$ has enough points, and so the theorem above ensures that $\ev : \E\to\vSh(\vpt(\E))$ is an equivalence, and as explained in Example \ref{ex:ultrasheaf}\ref{ex:ultrasheaf-ultracat} we recover Lurie's Theorem \ref{Thm:Lurie}.
\end{rmk}

The above theorem implies in particular that for any topos $\E$ the virtual ultracategory $\vpt(\E)$ is bounded; so $\vpt : \GTop\to\vUlt$ corestricts to a 2-functor $\vpt : \GTop\to\vUltb$. We now present the above results as a pseudoidempotent 2-adjunction.

\begin{cor}\label{cor:final-adj}
    There is a pseudoidempotent 2-adjunction
    \[\begin{tikzcd}
        \GTop && \vUltb
	    \arrow[""{name=0, anchor=center, inner sep=0}, "{\vpt}"', curve={height=18pt}, from=1-1, to=1-3]
	    \arrow[""{name=1, anchor=center, inner sep=0}, "\vSh"', curve={height=18pt}, from=1-3, to=1-1]
        \arrow["\vdash"{anchor=center, rotate=90}, draw=none, from=0, to=1]
    \end{tikzcd}\]
    with counit $\epsilon_{\E} : \vSh(\vpt(\E))\to\E$ given by the evaluation functor. The induced comonad is given by the restriction of a topos to all its points.
    In particular, the 2-category of topoi with enough points embeds reflectively in $\vUltb$.
\end{cor}
\begin{proof}
    From a bounded v-ultracategory $X$, there is a natural functor $\eta_X : X \to \vpt(\vSh(X))$ mapping an object $a$ of $X$ to the evaluation $\ev_a : \vSh(X)\to\Set$ and an ultraarrow $f : a\ultrato (b_s)_{s:\mu}$ to the natural transformation whose component at $A\in\vSh(X)$ is given by $A(f) : A(a)\to\int_{s:\mu}A(b_s)$.
    
    One can then check that $\eta$ and $\epsilon$ induce strict 2-natural transformations and that the two triangle identities are satisfied strictly.

    For the pseudoidempotency: if $X$ is a bounded v-ultracategory, then $\vSh(X)$ has always enough points, so Theorem \ref{Thm:thetrueone} ensures that $\ev : \vSh(X) \to \vSh(\vpt(\vSh(X)))$ is an equivalence, and so $\epsilon_{\vSh(X)}$ is an equivalence.
\end{proof}

\begin{rmk}
    Our proof makes use of the axiom of choice, for example, when dealing with the indexings. However, a significant part of the theory relies only on the ultrafilter principle; it is thus natural to wonder whether our reconstruction result really needs choice or if the ultrafilter principle could be enough.
\end{rmk}

\begin{rmk}
    This reflection of topoi with enough points inside bounded virtual ultracategories might be helpful to compute (co)limits of topoi with enough points. For example, together with Lemma \ref{lem:T1ample}, it implies that $\Sh(T_{\bullet}^{\amp})$ of Construction \ref{cons:T1} is indeed the groupoid kernel of $\Sh(T_0^{\amp})\to\E$ in the 2-category of topoi with enough points; in other words, the groupoid of amply indexed models is ``étale-complete'' for a topoi with enough points analogous version of étale-completeness \cite{Moerdijk}; and this can be easily adapted for other groupoids of indexed models.
\end{rmk}

\section{Conclusion}

Let us take a look at what we have accomplished. We defined a new notion categorifying topological spaces that we named ``virtual ultracategories''. We then show that any topos possesses such a structure on its points; if the topos is localic (resp. coherent), we recover the known notion of relational $\beta$-module (resp. ultracategory). We then proved a reconstruction result: \emph{a topos with enough points can be reconstructed by taking ultrasheaves on its virtual ultracategory of points}. This result categorifies the known fact that a topological space can be recovered from its relational $\beta$-module of points, and generalizes conceptual completeness (a coherent topos can be recovered from its ultracategory of points).

Our proof makes use of representations of topoi by topological groupoids; the two technical ingredients of our proof are given by the 0-dimensional case (Theorem \ref{Thm:étale-ucvg}) and a descent theorem for virtual ultracategories (Proposition \ref{prop:desc-vUlt}). We expect one could also have adapted the proof of Makkai \cite{MakkaiStone} or the one of Lurie \cite{Lurie} to this setting.

In this paper, we focused on proving the reconstruction theorem as we deemed it a clear and strong motivation for the introduction of virtual ultracategories. Of course, this new notion of categorified spaces opens the door to new directions of research, a few of which are listed below, that we leave to future work.

\begin{enumerate}
    \item How do virtual ultracategories relate to ionads? Ionads are another categorification of topological spaces introduced by Garner in \cite{IonadGarner} and revisited by Di Liberti in \cite{IonadIvan}.
    In particular we would like to compare our adjunction of Corollary \ref{cor:final-adj} with Di Liberti's adjunction \cite[Theorem 3.2.11]{IonadIvan}.
    \item In the same way that ultracategories can be seen as stacks over the category of compact Hausdorff spaces \cite[Section 4]{Lurie}, we expect that virtual ultracategories can be seen as stacks over $\TopSp$. This is related to the notion of \emph{profile} introduced in \cite[Section 3.2]{Ivan}.
    \item When is the virtual ultracategory of points of a topos $\E$ actually an ultracategory? In the 0-dimensional case, topological spaces whose v-ultracategory of points is an ultracategory are related to \emph{strongly sober} spaces, a notion which appears in the domain theory literature \cite[Definition VI-6.12]{StronglySober}. Could this help characterize coherent topoi among topoi with enough points? 
\end{enumerate}

\subsection*{Acknowledgments} I would like to thank my PhD advisor \emph{Peter Lumsdaine} for his guidance and his precious advice during the elaboration of this work. I am also grateful to \emph{Errol Yuksel} for his useful suggestions and for his careful reading of preliminary versions of this paper. I would finally like to thank \emph{Ivan Di Liberti} for the suggestion of the name ``virtual ultracategories'' and \emph{Josh Wrigley} for having taken the time to answer some questions I had about his theorem on topological groupoids.

\subsection*{Funding}
The author was supported by the Knut and Alice Wallenberg Foundation Project ``Type Theory for Mathematics and Computer Science'' (PI Thierry Coquand). 

\printbibliography

\end{document}